\documentclass[11pt,a4paper]{article}
\usepackage[T1]{fontenc}
\usepackage{lmodern}
\usepackage[margin=26mm]{geometry}
\usepackage{amsmath,amssymb,amsthm,mathtools}
\usepackage{enumitem}
\usepackage{mathrsfs}
\usepackage{indentfirst}
\usepackage{microtype}
\usepackage[colorlinks=true,linkcolor=blue,citecolor=blue,urlcolor=blue]{hyperref}
\numberwithin{equation}{section}
\newtheorem{theorem}{Theorem}[section]
\newtheorem{proposition}[theorem]{Proposition}
\newtheorem{lemma}[theorem]{Lemma}
\newtheorem{corollary}[theorem]{Corollary}
\theoremstyle{definition}

\newtheorem{assumption}{Assumption}

\theoremstyle{remark}
\newtheorem{remark}[theorem]{Remark}
\newcommand{\R}{\mathbb R}
\newcommand{\sgn}{\operatorname{sgn}}
\newcommand{\erfc}{\operatorname{erfc}}

\allowdisplaybreaks[2]
\setlist[enumerate]{label=(\roman*),leftmargin=2.2em}
\title{Time Delay of KPP Fronts Across a Semipermeable Barrier}
\author{Yantao Wang\\[0.3em]
\small\href{mailto:wangyt4455@icloud.com}{\texttt{wangyt4455@icloud.com}}}
\date{}

\begin{document}
\maketitle
\begin{abstract}
We study a KPP reaction--diffusion equation on the real line with a semipermeable barrier at the origin, where the flux is continuous and the density jump is proportional to the spatial derivative. With diffusion and linear growth normalized to one, let $\beta$ denote the barrier resistance. For initial data equal to one sufficiently far to the left and zero near the barrier and on the receiving side, we prove convergence on the receiving side to a minimal-speed wave with a unique phase in the Bramson frame. The difference between distant first-arrival times with and without the barrier therefore has a finite limit $\tau_\beta$, independent of the arrival level and continuous in $\beta$. For concave KPP reactions, this delay is strictly increasing in $\beta$ and satisfies $\tau_\beta=K\beta+o(\beta)$ as $\beta\downarrow0$, with $K>1/4$. Without the concavity assumption, the strong-barrier expansion is $\tau_\beta=\frac12\log\beta-\frac12\log\log\beta+C+o(1)$, where $C$ is determined by the homogeneous reference amplitude and the closed donor-side Neumann evolution. The analysis relates the sustained interface input to the front phase through leading-edge estimates, interior matching, and parameter-uniform moment bounds.
\end{abstract}

\section{Introduction and main results}\label{sec:introduction}

How long does a geographical barrier delay the spread of an advantageous allele, and how does this delay depend on the strength of the barrier? Pi\'alek and Barton~\cite{PB1997} investigated this question in a deterministic additive-selection model that reduces, after scaling, to a logistic reaction--diffusion equation with a semipermeable barrier. Their analytical approximations and numerical computations indicated logarithmic growth of the delay for strong barriers. Barton and Turelli~\cite{BT2011}, drawing on~\cite{PB1997}, contrasted this behavior with the possibility of stopping bistable waves at local barriers. Lou, Nagylaki and Ni~\cite[Section~3.3.3]{LNN2013} subsequently identified the study of traveling waves across a geographical barrier as an open problem, highlighting the numerical evidence that the asymptotic delay decreases as barrier transmissivity increases. These observations motivate a precise definition of the lasting delay and an analysis of its dependence on barrier resistance.

The interface description has its roots in migration--selection theory. Nagylaki~\cite{Nagylaki1976} distinguished the frequency jump caused by a localized geographical barrier from the gradient jump caused by a discontinuity in migration rate, and studied their effects on the maintenance of polymorphism. Barton and Bengtsson~\cite{BB1986} quantified barriers to genetic exchange between demes and across continuous habitats, focusing on barriers generated by selection at linked loci. The present model instead prescribes a geographical barrier through an interface law that preserves flux continuity and makes the exchange flux proportional to the density jump.

The mathematical theory of such interfaces includes Chen's study of barrier boundary value problems for parabolic and elliptic equations~\cite{Chen2001}. In bounded habitats, Chen~\cite{Chen2006} established a comparison principle and global existence for a reaction--diffusion interface problem, and studied how equilibrium stability depends on permeability. Wang and Su~\cite{WangSu2022} investigated global dynamics and the existence, stability, and profiles of spatially nonconstant equilibria. Subsequent work addressed positive equilibria under asymmetric transmission~\cite{MMS2023} and parameter dependence of stationary solutions in spatially heterogeneous population models~\cite{AB2023}. For propagation on the real line, Hamel, Lutscher and Zhang~\cite{HLZ2022} studied spreading and blocking in a two-patch model with KPP and bistable reactions, together with traveling-wave stability in appropriate bistable regimes. Their transmission law couples flux continuity to a fixed proportionality of the original density traces, whereas here the density jump is proportional to the flux. Our question concerns the lasting phase change after a front crosses this interface.

The homogeneous Fisher--KPP equation, originating in the studies of gene spread by Fisher~\cite{Fisher1937} and Kolmogorov, Petrovskii and Piskunov~\cite{KPP1937}, provides the reference dynamics. For the logistic reaction and step-like initial data, Bramson~\cite{Bramson1983} proved convergence to a translate of the minimal-speed wave in the frame $2t-\frac32\log t$. Hamel et al.\ \cite{HNRR2013} gave a PDE proof of the logarithmic correction, while Nolen--Roquejoffre--Ryzhik~\cite{NRR2017} established convergence to a single translate through leading-edge analysis. To identify a finite barrier-induced delay, one needs the constant phase relative to the homogeneous solution in addition to the spreading speed and logarithmic correction. Mytnik--Roquejoffre--Ryzhik~\cite{MRR2022} obtained precise small-data asymptotics for the phase in the logistic equation. Their result supplies a comparison estimate in our strong-barrier analysis, where transmission provides a weak, sustained input to the receiving side.

We study this delay for spatially homogeneous KPP reactions and step-like initial data supported on the donor side. We define it as the limit of the difference between distant first-arrival times for solutions with and without a barrier, starting from the same initial data. We prove convergence on the receiving side to a minimal-speed wave with a unique phase in the Bramson frame, which yields a finite delay independent of the arrival level. The delay is continuous in the resistance; under an additional concavity assumption, it is strictly increasing and admits a first-order weak-barrier expansion. For strong barriers, we refine the earlier logarithmic description by deriving a log-log correction and a constant determined by the homogeneous reference phase and the evolution on the limiting closed donor side. The proofs combine self-similar leading-edge analysis, interior matching, and parameter-uniform moment estimates to relate the sustained interface input to the asymptotic phase.

\subsection{Model, assumptions, and notation}
We consider the spread of an advantageous allele in a one-dimensional habitat separated by a semipermeable barrier at the origin. The regions $x<0$ and $x>0$ are the donor and receiving sides, respectively. Let $u(t,x)$ denote the allele frequency. Away from the barrier, migration is modeled by diffusion and selection by a reaction term $f$. By rescaling time and space, we normalize both the diffusion coefficient and the linear growth rate at zero to one, so that $f'(0)=1$. In these rescaled variables, the evolution for $0<\beta<\infty$ is governed by
\begin{equation}\label{eq:model}
 u_t=u_{xx}+f(u),\qquad t>0,\quad x\ne0,
 \qquad u(0,x)=u_0(x).
\end{equation}
We model the barrier as an interface with no absorption or accumulation, across which the diffusive flux is conserved. Its finite permeability allows the allele frequency to have a jump, with the transmitted flux proportional to the frequency difference across the interface, as in the geographical-barrier description of~\cite{Nagylaki1976,PB1997}. For a spatial function, write $[v]=v(0^+)-v(0^-)$. For a time-dependent function, write $[v](t)=v(t,0^+)-v(t,0^-)$, omitting $t$ when it is clear. We impose
\begin{equation}\label{eq:interface}
 u_x(t,0^-)=u_x(t,0^+),\qquad
 [u](t)=\beta u_x(t,0^\pm),\qquad t>0.
\end{equation}
The parameter $\beta>0$ measures the resistance to exchange: for a fixed frequency difference, a larger resistance produces a smaller flux magnitude. In the present scaling, it is the reciprocal of the transmissivity, denoted here by $\gamma_{\rm tr}$, in the convention of~\cite{LNN2013,WangSu2022}, so $\beta=\gamma_{\rm tr}^{-1}$. The subscript distinguishes transmissivity from the cutoff exponents used below.

When the dependence on the resistance needs to be explicit, we write the solution as $u^\beta$. The case $\beta=0$ represents the absence of a barrier. We denote by $u^0$ the homogeneous full-line solution with the same initial data. Formally, the opposite limit $\beta\to\infty$ corresponds to an impermeable barrier, with zero flux on each side.

We impose the following KPP assumptions on the reaction term.
\begin{assumption}\label{hyp:f}
The function $f$ satisfies
\begin{equation}\label{eq:kpp}
 f\in C^2([0,1]),\quad f(0)=f(1)=0,\quad
 f'(0)=1,\quad f'(1)<0,\quad 0<f(s)\le s\quad(0<s<1).
\end{equation}
\end{assumption}
The states $0$ and $1$ represent absence and fixation of the advantageous allele, respectively. The positivity of $f$ favors its increase at intermediate frequencies, while the KPP bound ensures that the reaction term does not exceed its linearization at zero. The condition $f'(1)<0$ makes fixation linearly stable under the reaction dynamics.

The invasion is initiated by step-like data on the donor side.
\begin{assumption}\label{hyp:data}
The initial function $u_0$ is measurable, and there exist $a<b<0$ such that
\begin{equation}\label{eq:data}
 0\le u_0\le1,\qquad u_0(x)=1\quad(x\le a),\qquad
 u_0(x)=0\quad(x\ge b).
\end{equation}
\end{assumption}
Thus, the allele is fixed sufficiently far to the left and absent near the barrier and throughout the receiving side. No continuity or monotonicity is required in the transition interval $(a,b)$.

Only for strict ordering of the delay with respect to the resistance and the first-order weak-barrier expansion do we impose the following additional concavity condition.
\begin{assumption}\label{hyp:concave}
The reaction term satisfies $f''(s)\le0$ for $0\le s\le1$.
\end{assumption}
The logistic reaction $f(s)=s(1-s)$ satisfies this condition. Assumptions~\ref{hyp:f}--\ref{hyp:data} suffice for phase convergence, continuity of the far-field delay, the strong-barrier expansion, and the finite-time first variation at zero resistance. Assumption~\ref{hyp:concave} is used additionally for strict ordering and for the far-field first-order expansion, including the bound on its coefficient; strict concavity is not required.

To describe the front after it has crossed the barrier, we use the minimal-speed traveling wave of the corresponding homogeneous equation as a reference profile. Under Assumption~\ref{hyp:f} and our normalization, monotone fronts exist precisely for speeds $c\ge2$, and each profile is unique up to translation; see~\cite[discussion preceding Theorem~1.2]{HNRR2013}. The minimal wave speed is therefore $2$. Writing $z$ for the coordinate in a frame moving to the right at speed $2$, let $U_*(z)$ denote the decreasing wave profile satisfying
\begin{equation}\label{eq:wave}
 U_*''+2U_*'+f(U_*)=0,\quad
 U_*(-\infty)=1,\quad U_*(+\infty)=0.
\end{equation}
We fix the translation by requiring $U_*(0)=1/2$. Its leading edge has the asymptotic form
\begin{equation}\label{eq:wave-tail}
 U_*(z)=(A_*z+B_*+o(1))e^{-z}\qquad(z\to+\infty),\qquad A_*>0.
\end{equation}
For clarity, this tail expansion also follows directly from the wave ODE.
Linearization at $U_*=0$ gives $U_*,U_*'=O(e^{-(1-\delta)z})$ for
any $0<\delta<1/2$. With $V_*=e^zU_*$, the equation becomes
$V_*''=e^z[U_*-f(U_*)]\ge0$. Its right-hand side has an integrable first
moment at $+\infty$, since $s-f(s)=O(s^2)$. Integrating twice gives
$V_*=A_*z+B_*+o(1)$. Moreover, $V_*'(-\infty)=0$ and
$U_*-f(U_*)$ is positive near $U_*=1$, so $A_*>0$.
The coefficient $A_*$ will be used below to relate the front phase to its
leading-edge amplitude.

For the homogeneous evolution under Assumptions~\ref{hyp:f}--\ref{hyp:data}, the front position includes the Bramson logarithmic correction. We therefore introduce
\begin{equation}\label{eq:m}
 m(t)=2t-\frac32\log t\qquad(t>0),
\end{equation}
as the reference position. The phase relative to this frame will encode the effect of a finite-resistance barrier.

To measure the resulting delay, we record when the allele frequency first reaches a prescribed level at a point on the receiving side. For $\beta\ge0$, $\eta\in(0,1)$, and $R>0$, define
\begin{equation}\label{eq:arrival}
 T_\eta^\beta(R)=\inf\{t>0:u^\beta(t,R)\ge\eta\},
\end{equation}
with the convention $\inf\varnothing=+\infty$. The difference $T_\eta^\beta(R)-T_\eta^0(R)$ compares arrival at the same level and location for the solutions with and without a barrier. The next theorem establishes the far-field delay and its relation to the asymptotic phase.

\subsection{Main results}
Throughout this subsection, $f$ and $u_0$ are fixed, and $\beta$ varies. We first establish the limiting arrival-time difference and its continuity; positivity and strict ordering are obtained under the additional concavity assumption.

\begin{theorem}\label{thm:phase}
Suppose that Assumptions~\ref{hyp:f}--\ref{hyp:data} hold. For every $\beta\in[0,\infty)$, there is a unique $\xi_\beta\in\R$ such that, as $t\to\infty$,
\begin{equation}\label{eq:phase-limit}
 \sup_{x\ge0}\big|u^\beta(t,x)-U_*(x-m(t)-\xi_\beta)\big|\longrightarrow0.
\end{equation}
At $x=0$, the value of $u^\beta$ in~\eqref{eq:phase-limit} is understood as the right-hand trace $u^\beta(t,0^+)$. Define the leading-edge amplitude by $\alpha_\beta=A_*e^{\xi_\beta}>0$. For every $\eta\in(0,1)$, the following limit exists, is finite, and is independent of $\eta$:
\begin{equation}\label{eq:delay}
 \tau_\beta:=\lim_{R\to\infty}
 \big(T_\eta^\beta(R)-T_\eta^0(R)\big)
 =\frac{\xi_0-\xi_\beta}{2}
 =\frac12\log\frac{\alpha_0}{\alpha_\beta}.
\end{equation}
The maps $\beta\mapsto\alpha_\beta$ and $\beta\mapsto\tau_\beta$ are continuous on $[0,\infty)$, and $\tau_0=0$.
\end{theorem}

By~\eqref{eq:wave-tail}, $\alpha_\beta$ is the coefficient of $ze^{-z}$ in the tail of $U_*(z-\xi_\beta)$ and is independent of the translation used to normalize $U_*$. Our sign convention makes a larger phase $\xi_\beta$, and hence a larger amplitude $\alpha_\beta$, correspond to a front farther to the right. The asymptotic spatial lag relative to the homogeneous front is $\xi_0-\xi_\beta=2\tau_\beta$. We call $\tau_\beta$ the far-field time delay; under Assumptions~\ref{hyp:f}--\ref{hyp:data}, this denotes the signed limiting arrival-time difference. Its positivity for $\beta>0$ follows under concavity from the next theorem.

\begin{theorem}\label{thm:monotonicity}
Suppose that Assumptions~\ref{hyp:f}--\ref{hyp:concave} hold. The leading-edge amplitude is strictly decreasing in the resistance, whereas the far-field delay is strictly increasing. More precisely,
\begin{equation}\label{eq:strict-order}
 0\le\beta_1<\beta_2<\infty
 \quad\Longrightarrow\quad
 \alpha_{\beta_1}>\alpha_{\beta_2},\qquad
 \tau_{\beta_1}<\tau_{\beta_2}.
\end{equation}
Since $\tau_0=0$, it follows that $\tau_\beta>0$ for every $\beta>0$.
\end{theorem}

We now quantify the response near zero resistance through a first-order expansion.

\begin{theorem}\label{thm:small-beta}
Suppose that Assumptions~\ref{hyp:f}--\ref{hyp:concave} hold. The far-field delay is right differentiable at zero resistance and satisfies
\begin{equation}\label{eq:small-delay}
 \tau_\beta=K\beta+o(\beta)\qquad(\beta\downarrow0),\qquad
 K=\tau'_+(0)>\frac14.
\end{equation}
The leading-edge amplitude also has a right derivative at zero, and
\[
 \dot\alpha_0:=\lim_{\beta\downarrow0}
 \frac{\alpha_\beta-\alpha_0}{\beta},\qquad
 K=-\frac{\dot\alpha_0}{2\alpha_0}.
\]
The right first variation
\[
 w=\left.\partial_\beta u^\beta\right|_{\beta=0^+}
   =\lim_{\beta\downarrow0}\frac{u^\beta-u^0}{\beta}
\]
exists, with convergence uniform on each finite time interval in $L^\infty(\R\setminus\{0\})$, and satisfies
\begin{equation}\label{eq:first-variation-main}
 \begin{cases}
 w_t=w_{xx}+f'(u^0)w,&t>0,\quad x\ne0,\\
 [w](t)=u_x^0(t,0),&t>0,\\
 [w_x](t)=0,&t>0,\\
 w(0,x)=0,&x\ne0.
 \end{cases}
\end{equation}
The amplitude derivative is recovered from the leading edge of this first variation: for every $0<\theta<1/3$,
\begin{equation}\label{eq:alpha-derivative-main}
 \dot\alpha_0=\lim_{t\to\infty}
 \frac{e^{t^\theta}w(t,m(t)+t^\theta)}{t^\theta}.
\end{equation}
\end{theorem}

The coefficient $K$ depends on $f$ and $u_0$ through the homogeneous solution and its linearized response. The finite-time first variation exists under Assumptions~\ref{hyp:f}--\ref{hyp:data}; concavity is used here to pass to the far-field amplitude response and prove $K>1/4$. Proposition~\ref{small:strict-coefficient} identifies the positive correction to $1/4$ through the response of the mean of the reflected donor-side and receiving-side frequencies.

For the strong-barrier limit, the relevant boundary history comes from the closed donor-side evolution. Reflecting this problem onto the positive half-line, so that $x>0$ corresponds to the original location $-x$, let $p_N$ solve
\begin{equation}\label{eq:donor}
 \begin{cases}
 (p_N)_t=(p_N)_{xx}+f(p_N),&t>0,\quad x>0,\\
 (p_N)_x(t,0)=0,&t>0,\\
 p_N(0,x)=u_0(-x),&x>0.
 \end{cases}
\end{equation}
The Neumann condition expresses the absence of exchange through the closed barrier. We denote the allele frequency at this boundary by $h(t)$ and define its exponentially weighted time integral by
\begin{equation}\label{eq:donor-mass}
 h(t)=p_N(t,0),\qquad I=\int_0^\infty e^{-2t}h(t)\,dt.
\end{equation}
Under Assumptions~\ref{hyp:f}--\ref{hyp:data}, the strong maximum principle gives $0<h(t)<1$ for $t>0$, and hence $0<I<1/2$. The exponential weight in $I$ emphasizes the earlier part of this boundary history.

\begin{theorem}\label{thm:large-beta}
Suppose that Assumptions~\ref{hyp:f}--\ref{hyp:data} hold. As $\beta\to\infty$, the leading-edge amplitude satisfies
\begin{equation}\label{eq:large-amplitude}
 \alpha_\beta=\frac{I}{\sqrt\pi}\frac{\log\beta}{\beta}(1+o(1)).
\end{equation}
Consequently, the far-field delay satisfies
\begin{equation}\label{eq:large-delay}
 \tau_\beta=\frac12\log\beta-\frac12\log\log\beta
 +\frac12\log\frac{\sqrt\pi\alpha_0}{I}+o(1).
\end{equation}
\end{theorem}

Expansion~\eqref{eq:large-delay} specifies the log-log correction and the constant accompanying the leading logarithmic growth. The constant retains information about both the homogeneous reference solution, through $\alpha_0$, and the closed donor-side boundary history, through $I$. Concavity is not required.

\begin{remark}
For each finite resistance $\beta$, the far-field limit $R\to\infty$ in~\eqref{eq:delay} is taken first to define $\tau_\beta$; Theorem~\ref{thm:large-beta} then describes its behavior as $\beta\to\infty$. If the barrier is made completely impermeable, the receiving side remains at zero under Assumption~\ref{hyp:data}, so no positive level is reached there in finite time. Thus, the expansion concerns the far-field delay for finite resistances and does not directly give the establishment time at a fixed spatial point.
\end{remark}

\begin{remark}
In the deterministic additive-selection model of~\cite{PB1997}, the barrier strength is $B=[p]/p_x$; after scaling, this gives the resistance in~\eqref{eq:interface}. Their Appendix~A estimates establishment at the barrier from linearized growth. The numerical delay in their Figure~1A is computed from the difference in total allele numbers between solutions with and without a barrier after the wave has crossed, while Figure~1B follows the evolving difference in wave position. Definition~\eqref{eq:delay} instead specifies a limit of distant first-arrival times; Theorem~\ref{thm:phase} identifies this limit with the asymptotic phase difference divided by the wave speed.
\end{remark}

Three parameter-uniform estimates connect the finite-time dynamics to these asymptotic results. The absolute first-moment bound~\eqref{front:uniform-moment} transfers continuity to the limiting amplitude; positivity is proved separately for each finite resistance. The difference-quotient estimate~\eqref{small:moment-uniform-error} justifies differentiation at zero resistance. For large resistance, the shifted adjoint-moment bound~\eqref{large:uniform-first-moment} controls the sustained input from above, while a comparison based on the small-data amplitude law of Mytnik--Roquejoffre--Ryzhik~\cite{MRR2022} gives the matching lower bound.

We first establish the transmission heat kernel and comparison properties, then derive the leading-edge estimates and match them to the interior wave profile to prove phase convergence. The remaining sections treat continuity and strict monotonicity in the resistance, followed by the weak- and strong-barrier limits.

\section{Transmission heat kernel, well-posedness, and reflection comparison}\label{sec:preliminaries}

This section establishes well-posedness and comparison for the transmission problem, together with the reflection properties and receiving-side estimates needed below.

Let $G_t(x)=(4\pi t)^{-1/2}e^{-x^2/(4t)}$ be the full-line heat kernel. An explicit kernel for diffusion across a semipermeable interface is given by Bressloff~\cite[Section~II.B, equations~(2.28a)--(2.28b)]{Bressloff2024}. Taking $D=1$ and $\kappa_0=2/\beta$ in that formula gives the following equivalent integral representation for our interface law. The associated Markov semigroup and its generator are characterized by Lejay~\cite[Section~4, Propositions~1--2]{Lejay2016}; his normalization corresponds to $\kappa=2/\beta$ and the time change $t\mapsto2t$. For $\beta>0$, $t>0$, and $x,y\ne0$, the kernel is
\begin{equation}\label{eq:barrier-kernel}
 \begin{split}
 P_\beta(t,x,y)={}&G_t(x-y)
 +\sgn(x)\sgn(y)\left[
 G_t(|x|+|y|)
 -\int_0^\infty e^{-r}
 G_t(|x|+|y|+\beta r/2)\,dr\right].
 \end{split}
\end{equation}
Values at the interface are understood as one-sided limits. For $\beta=0$, we set $P_0(t,x,y)=G_t(x-y)$.

For completeness, write
\[
 H_\beta(t,a)=\int_0^\infty e^{-r}G_t(a+\beta r/2)\,dr,
 \qquad a\ge0.
\]
Completing the square yields
\[
 H_\beta(t,a)=\frac1\beta
 e^{2a/\beta+4t/\beta^2}
 \erfc\!\left(\frac{a}{2\sqrt t}+\frac{2\sqrt t}{\beta}\right),
\]
recovering the formula in~\cite{Bressloff2024}, with reflection symmetry covering negative source points. Integration by parts also gives
\[
 \partial_aH_\beta(t,a)=\frac2\beta\bigl(H_\beta(t,a)-G_t(a)\bigr),
\]
which verifies both interface conditions in the $x$ variable in \eqref{eq:barrier-kernel}. The heat equation away from the interface follows by differentiation. Since $G_t$ is decreasing on $[0,\infty)$, the correction to $G_t(x-y)$ is nonnegative when $x$ and $y$ have the same sign; when they have opposite signs, $P_\beta(t,x,y)=H_\beta(t,|x|+|y|)>0$. Thus the kernel is nonnegative. Moreover, the correction is odd in $y$ for each fixed $x$, so
\[
 \int_\R P_\beta(t,x,y)\,dy=1.
\]
The symmetry $P_\beta(t,x,y)=P_\beta(t,y,x)$ also gives conservation of mass. Reflection decomposes the linear transmission problem into half-line Neumann and Robin problems (Neumann and Dirichlet when $\beta=0$), whose uniqueness gives the semigroup property. Together with positivity and preservation of constants, this shows that the kernels define a positive contraction semigroup on bounded measurable functions, consistent with the characterization in~\cite{Lejay2016}.

\begin{proposition}\label{prop:well-posedness}
Under Assumption~\ref{hyp:f}, let $u_0$ be measurable with $0\le u_0\le1$, and fix $0\le\beta<\infty$. There is a unique global bounded mild solution $u^\beta$, characterized by
\begin{equation}\label{eq:mild}
 u^\beta(t,x)=\int_\R P_\beta(t,x,y)u_0(y)\,dy
 +\int_0^t\int_\R P_\beta(t-s,x,y)f(u^\beta(s,y))\,dy\,ds.
\end{equation}
It satisfies $0\le u^\beta\le1$ and is classical on each side for $t>0$, with one-sided traces satisfying \eqref{eq:interface}. At $\beta=0$, it is the classical homogeneous full-line solution for positive times. As $t\downarrow0$, $u^\beta(t,\cdot)\to u_0$ locally in $L^p(\R)$ for every $1\le p<\infty$, and pointwise at every continuity point of $u_0$ away from the interface.

For the same fixed $\beta$ and reaction $f$, comparison holds for any two measurable initial data taking values in $[0,1]$: if $u_0\le v_0$ almost everywhere, then their solutions satisfy $u^\beta(t,x)\le v^\beta(t,x)$ for $t>0$, including the corresponding one-sided interface traces.
\end{proposition}

\begin{proof}
Extend $f$ to a globally Lipschitz function on $\R$. The contraction property of the linear semigroup makes the right-hand side of \eqref{eq:mild} a contraction on bounded measurable functions on a sufficiently short time interval. This gives local existence and uniqueness without requiring continuity of the initial data.

To obtain comparison, choose a Lipschitz constant $L$ for the extension, so that $s\mapsto f(s)+Ls$ is nondecreasing. In the equivalent integral formulation with the damped kernel $e^{-Lt}P_\beta(t,x,y)$ and reaction $f(u)+Lu$, positivity of the kernel and Picard iteration preserve the order of the initial data. Comparison with the constant solutions $0$ and $1$ gives $0\le u^\beta\le1$, and hence global continuation and independence of the chosen extension of $f$.

Kernel smoothing, followed by parabolic regularity for the folded Neumann
and Robin problems, gives positive-time classical regularity and the
transmission traces; see~\cite{Lieberman1996} for the scalar boundary
regularity tools. The explicit kernel supplies the interface structure used here. Finally, the kernel has the initial approximation property locally in $L^p$ and at continuity points away from the interface, while the reaction integral in \eqref{eq:mild} is $O(t)$ in $L^\infty$. This proves the asserted initial convergence.
\end{proof}

We use the scalar parabolic maximum principle, its strong form, and the
boundary point lemma as in~\cite[Chapter~II]{Lieberman1996}.
We now reflect the donor side onto the positive half-line, so that both sides can be compared on the same domain. The average and difference of the reflected profiles separate the interface law into Neumann and Robin conditions. For $x>0$, set
\begin{equation}\label{eq:fold}
 p_\beta(t,x)=u^\beta(t,-x),\quad q_\beta(t,x)=u^\beta(t,x),
 \quad S_\beta=\frac{p_\beta+q_\beta}{2},\quad
 D_\beta=\frac{p_\beta-q_\beta}{2}.
\end{equation}
For $\beta>0$, the interface conditions give
\begin{equation}\label{eq:fold-bc}
 (S_\beta)_x(t,0)=0,\qquad
 (D_\beta)_x(t,0)=\frac2\beta D_\beta(t,0).
\end{equation}
At zero resistance, $(S_0)_x(t,0)=0$ and $D_0(t,0)=0$. For every finite $\beta\ge0$, the nonlinear equations become
\begin{equation}\label{eq:fold-pde}
 \begin{split}
 (S_\beta)_t&=(S_\beta)_{xx}
 +\frac{f(S_\beta+D_\beta)+f(S_\beta-D_\beta)}2,\\
 (D_\beta)_t&=(D_\beta)_{xx}
 +\frac{f(S_\beta+D_\beta)-f(S_\beta-D_\beta)}2.
 \end{split}
\end{equation}
\begin{lemma}\label{lem:reflection-comparison}
Under Assumptions~\ref{hyp:f}--\ref{hyp:data}, fix $0\le\beta<\infty$. Then $p_\beta(t,x)>q_\beta(t,x)$ for $t>0$ and $x>0$. For $\beta>0$, this inequality also holds at $x=0$, and
\begin{equation}\label{eq:reflection}
 j_\beta(t):=-q_{\beta,x}(t,0)
 =\frac{p_\beta(t,0)-q_\beta(t,0)}\beta>0,
 \qquad t>0.
\end{equation}
Here $j_\beta$ is oriented from the donor side to the receiving side, so $[u^\beta](t)=-\beta j_\beta(t)$ for $\beta>0$. For $\beta=0$, the interface flux satisfies $j_0(t):=-u_x^0(t,0)>0$. Moreover, for every finite $\beta\ge0$,
\begin{equation}\label{eq:right-monotone}
 q_{\beta,x}(t,x)<0\qquad(t>0,x>0).
\end{equation}
\end{lemma}

\begin{proof}
Define
\[
 c_\beta(t,x)=\int_0^1 f'\!\left(q_\beta(t,x)
 +s\bigl(p_\beta(t,x)-q_\beta(t,x)\bigr)\right)\,ds.
\]
Then $(D_\beta)_t=(D_\beta)_{xx}+c_\beta D_\beta$, with
$|c_\beta|\le\|f'\|_{L^\infty([0,1])}$ and
$D_\beta(0,x)=u_0(-x)/2\ge0$ not identically zero. Robin comparison for $\beta>0$, or Dirichlet comparison for $\beta=0$, followed by the strong maximum principle gives $D_\beta(t,x)>0$ for $t>0$ and $x>0$. If $\beta>0$ and $D_\beta(t,0)=0$ at a positive time, the boundary Hopf lemma would give $(D_\beta)_x(t,0)>0$, contradicting \eqref{eq:fold-bc}. Thus $D_\beta(t,0)>0$, and the interface law yields \eqref{eq:reflection}. When $\beta=0$, the Hopf lemma instead gives
$j_0(t)=(D_0)_x(t,0)>0$.

For the receiving-side monotonicity, set $v=q_{\beta,x}$. At positive times,
\[
 v_t=v_{xx}+f'(q_\beta)v,\qquad x>0,
 \qquad v(t,0)=-j_\beta(t)<0.
\]
The zero initial condition for this derivative is justified directly from the kernel, even though $u_0$ is only measurable. Indeed, differentiation of \eqref{eq:barrier-kernel} gives
$\sup_{x>0}\int_\R|\partial_xP_\beta(t,x,y)|\,dy\le Ct^{-1/2}$.
Since $u_0$ vanishes for $y\ge b$ with $b<0$, the initial term and the bounded reaction term in \eqref{eq:mild} therefore give
\[
 \|q_{\beta,x}(t,\cdot)\|_{L^\infty(0,\infty)}
 \le Ct^{-1/2}e^{-c/t}+C\sqrt t\longrightarrow0
 \qquad(t\downarrow0),
\]
where $c>0$. With $M=\|f'\|_{L^\infty([0,1])}$ and $v^+=\max\{v,0\}$, comparison on $[\varepsilon,T]\times(0,\infty)$ gives
\[
 v(t,x)\le e^{M(t-\varepsilon)}
 \|v(\varepsilon,\cdot)^+\|_{L^\infty(0,\infty)},
 \qquad \varepsilon\le t\le T.
\]
Letting $\varepsilon\downarrow0$ yields $v\le0$. The strong maximum principle and the strictly negative boundary values give $v<0$ in the interior. Neither concavity of $f$ nor monotonicity of $u_0$ is used.
\end{proof}

We conclude with estimates needed for the leading-edge analysis. The function
$W(t,x)=e^t\erfc(x/(2\sqrt t))$ satisfies $W_t=W_{xx}+W$ on $x>0$, has zero initial value there, and obeys $W(t,0)=e^t\ge q_\beta(t,0)$. Since $f(q_\beta)\le q_\beta$, half-line comparison yields the parameter-uniform Gaussian bound
\begin{equation}\label{eq:uniform-gaussian}
 q_\beta(t,x)\le e^t\erfc\!\left(\frac{x}{2\sqrt t}\right),\qquad t>0,\quad x>0.
\end{equation}
For $z\ge-t^\rho$ with $0<\rho<1/3$, write $x=m(t)+z=2t+z-\frac32\log t$. Then $x\ge t$ for all sufficiently large $t$. The bound $\erfc(r)\le e^{-r^2}/(\sqrt\pi r)$ for $r>0$ gives
\[
 e^t\erfc\!\left(\frac{x}{2\sqrt t}\right)
 \le Ct^{-1/2}e^{t-x^2/(4t)}
 =Ct e^{-z}\exp\!\left(-\frac{(z-\frac32\log t)^2}{4t}\right)
 \le Ct e^{-z}.
\]
Hence, throughout this range,
\begin{equation}\label{eq:uniform-leading-bound}
 q_\beta(t,m(t)+z)\le Ct e^{-z},
\end{equation}
where $C$ is independent of all finite $\beta\ge0$. To separate the linear growth from the nonlinear remainder in the next section, set $g(s)=s-f(s)$. Assumption~\ref{hyp:f}, in particular $f'(0)=1$, gives
\begin{equation}\label{eq:g}
 0\le g(s)\le C s^2,\qquad |g'(s)|\le Cs\qquad(0\le s\le1).
\end{equation}
\section{Uniform leading-edge estimates and amplitudes}\label{front:section}

The truncated self-similar approach in~\cite[Sections~3 and~5]{NRR2017}
motivates the linear tool developed here. That paper treats the logistic
equation; below we prove the kernel estimates and parameter-uniform bounds
needed for our general KPP reaction, including the signed-solution version. The estimates for the nonlinear solution use only Assumptions~\ref{hyp:f}--\ref{hyp:data}, through \eqref{eq:uniform-leading-bound} and \eqref{eq:g}. The signed-solution corollary holds under its explicit coefficient assumptions; concavity is used later to verify those assumptions for the weak-barrier difference quotient.

\subsection{The self-similar kernel on the half-line}
The exponential weighting and self-similar change of variables used at the leading edge produce the operator below. We first establish its half-line kernel estimates, before applying them to the nonlinear solution in the next subsection. Let
\[
 \mathcal L=-\partial_{yy}-\frac y2\partial_y-1,
 \qquad \phi(y)=ye^{-y^2/4},\qquad y>0.
\]
Here $r>0$ denotes an increment of self-similar time. With homogeneous Dirichlet boundary conditions, the kernel of $S(r)=e^{-r\mathcal L}$ follows from the half-line heat kernel by the spatial scaling $y\mapsto e^{r/2}y$, the heat-time change $e^r-1$, and multiplication by $e^r$:
\begin{equation}\label{front:kernel}
 K(r,y,s)=\frac{e^r}{\sqrt{4\pi(e^r-1)}}
 \left\{
 e^{-(e^{r/2}y-s)^2/[4(e^r-1)]}
 -e^{-(e^{r/2}y+s)^2/[4(e^r-1)]}
 \right\}.
\end{equation}
The Poisson kernel for boundary data is
\begin{equation}\label{front:poisson}
 B(r,y)=\partial_sK(r,y,0)
 =\frac{y}{2\sqrt\pi(1-e^{-r})^{3/2}}
 e^{-y^2/[4(1-e^{-r})]}.
\end{equation}
For a continuous boundary value $b$, its contribution to the solution is
\[
 \mathcal B b(\sigma,y)=\int_{\sigma_0}^{\sigma}
 B(\sigma-s,y)b(s)\,ds,\qquad y>0.
\]
This integral is defined in the interior and extended to $y=0$ by its trace, which equals $b(\sigma)$ for $\sigma>\sigma_0$. Substitution of $y=0$ under the integral is not valid, despite $B(r,0)=0$ for each $r>0$.

Write
\[
 X_\theta=\{h\in C([0,\infty)):
       \|h\|_\theta:=\sup_{y\ge0}e^{\theta y^2}|h(y)|<\infty\},
 \qquad 0<\theta<\frac14,
\]
and define the projection
\begin{equation}\label{front:projection}
 \Pi h=\frac{\phi}{2\sqrt\pi}\int_0^\infty sh(s)\,ds.
\end{equation}
This is the projection onto the zero mode, since
$\mathcal L\phi=0$ and $\int_0^\infty y\phi(y)\,dy=2\sqrt\pi$.
Direct integration of the kernel gives
\[
 \int_0^\infty yK(r,y,s)\,dy=s,\qquad
 \int_0^\infty yB(r,y)\,dy=1.
\]
Thus the homogeneous evolution preserves the first moment, $S(r)\Pi=\Pi S(r)=\Pi$, while the Poisson kernel transmits boundary input into that moment.
The kernel \eqref{front:kernel} gives directly
\begin{align}
 \|S(r)h\|_\theta&\le C_\theta\|h\|_\theta,
       &&r>0,\label{front:kernel-bounded}\\
 \|(S(r)-\Pi)h\|_\theta&\le C_\theta e^{-r}\|h\|_\theta,
       &&r\ge1,\label{front:kernel-gap}\\
 \|S(r)\partial_yh\|_\theta
       &\le C_\theta(1+r^{-1/2})\|h\|_\theta,
       &&r>0.\label{front:kernel-derivative}
\end{align}
In the last estimate, the derivative is interpreted as
$S(r)\partial_yh=-\int_0^\infty\partial_sK(r,\cdot,s)h(s)\,ds$;
the boundary term vanishes because $K(r,y,0)=0$. Its stable component satisfies
\begin{equation}\label{front:kernel-derivative-gap}
 \|(I-\Pi)S(r)\partial_yh\|_\theta
 \le C_\theta e^{-r}\|h\|_\theta,\qquad r\ge1.
\end{equation}
These estimates follow without abstract spectral theory. Writing the kernel as
\[
 \frac{e^{r/2}}{\sqrt\pi\sqrt{1-e^{-r}}}
 e^{-[y^2+e^{-r}s^2]/[4(1-e^{-r})]}
 \sinh\!\left(\frac{e^{-r/2}ys}{2(1-e^{-r})}\right),
\]
we extract $ys e^{-y^2/4}/(2\sqrt\pi)$ for $r\ge1$ and estimate the
remainder by Gaussian integration. More precisely, if $R_r(y,s)$ denotes the kernel minus this leading term, then
\[
 \sup_{y\ge0}e^{\theta y^2}\int_0^\infty
 \bigl(|R_r(y,s)|+|\partial_sR_r(y,s)|\bigr)e^{-\theta s^2}\,ds
 \le C_\theta e^{-r},\qquad r\ge1.
\]
This yields both \eqref{front:kernel-gap} and \eqref{front:kernel-derivative-gap}.
For $0<r\le1$, we use Gaussian convolution and its first derivative.
Similarly,
\begin{equation}\label{front:poisson-estimates}
 \left\|B(r,\cdot)-\frac{\phi}{2\sqrt\pi}\right\|_\theta
       \le C_\theta e^{-r}\quad(r\ge1),\qquad
 \sup_{y>0}e^{\theta y^2}
 \left|\int_0^1 B(r,y)b(r)\,dr\right|
       \le C_\theta\|b\|_{L^\infty(0,1)}.
\end{equation}
The second estimate holds for bounded measurable $b$ and follows from the heat Poisson integral over the most recent unit interval; it does not assert a continuous boundary trace for such data. It does not require the instantaneous norm of $B$ to be integrable:
that norm is of order $r^{-1}$ as $r\downarrow0$, so time integration must precede
the spatial supremum.

We now apply these estimates to truncated leading-edge equations with decaying drift, potential, and boundary input. The first part supplies a Gaussian bound and boundary control for a comparison function; the second identifies the limiting amplitude for a possibly signed solution. We give the proof for the precise bounds used here.

\begin{lemma}\label{lem:linear-front}
Fix $0<\rho<1/2$, and set
\[
 a_0=\frac12-\rho,
 \qquad d(\sigma)=e^{-a_0\sigma},
 \qquad \theta_0(\sigma)=\rho d(\sigma)+\frac32e^{-\sigma/2}.
\]
Let $\sigma_0\ge0$, and consider solutions for $\sigma\ge\sigma_0$, $y\ge0$, continuous up to the initial and lateral boundaries and classical in the interior. The initial and boundary values are compatible at $(\sigma_0,0)$, and the solutions are bounded in the respective weighted spaces on every finite time interval.
\begin{enumerate}
\item Suppose that $Q$ solves
\[
 Q_\sigma+\mathcal LQ+\theta_0Q_y=0,\qquad
 Q(\sigma,0)=h(\sigma),
\]
with $Q(\sigma_0,\cdot)\in X_{1/5}$, $h\in C^1([\sigma_0,\infty))$, and
$|h|+|h'|\le M\exp(-c e^{\rho\sigma})$ for fixed $M,c>0$.
If $\sigma_0$ is sufficiently large in terms of $\rho$, then
\begin{equation}\label{front:linear-upper}
 \sup_{\sigma\ge\sigma_0}\|Q(\sigma)\|_{1/5}\le C,
 \qquad
 \sup_{\sigma\ge\sigma_0+1}\sup_{0\le y\le1}|Q_y(\sigma,y)|\le C.
\end{equation}
If the initial and boundary data of $Q$ are nonnegative, then $Q\ge0$, and
$Q(\sigma,2d(\sigma))\le Cd(\sigma)$ for $\sigma\ge\sigma_0+1$.

\item Suppose that $P$, which may change sign, satisfies
\begin{equation}\label{front:linear-positive-cut}
 P_\sigma+\mathcal LP+\vartheta(\sigma)P_y+A(\sigma,y)P=0,
 \qquad P(\sigma,0)=b(\sigma),
\end{equation}
where $\vartheta$, $A$, and $b$ are continuous and, for fixed $M,c>0$,
\[
 |\vartheta(\sigma)|\le Md(\sigma),\quad
 |b(\sigma)|\le Md(\sigma),\quad
 \|A(\sigma,\cdot)\|_\infty
       \le M e^{M\sigma-c e^{\rho\sigma}},\quad
 \sup_{\sigma\ge\sigma_0}\|P(\sigma)\|_{1/6}\le M.
\]
Then there is a real number $\alpha$ such that
\begin{align}
 J(\sigma)&:=\int_0^\infty yP(\sigma,y)\,dy
       \longrightarrow2\sqrt\pi\alpha,
       & |J(\sigma)-2\sqrt\pi\alpha|&\le Cd(\sigma),
       \label{front:linear-moment}\\
 \|P(\sigma)-\alpha\phi\|_{1/6}&\le Cd(\sigma).
       \label{front:linear-profile}
\end{align}
\end{enumerate}
In part~(i), $C$ depends only on $\rho$, $M$, $c$, and an upper bound for $\|Q(\sigma_0)\|_{1/5}$. In part~(ii), it depends only on $\rho$, $\sigma_0$, $M$, and $c$; no derivative bound for $b$ enters the estimates. Thus families with a fixed starting time and common bounds have common constants.
\end{lemma}

\begin{proof}
Choose a fixed smooth, compactly supported function $\chi$ with $\chi(0)=1$,
and lift the boundary data in the first part by $h(\sigma)\chi(y)$.
The function $V=Q-h\chi$ has homogeneous boundary data and satisfies
\[
 V_\sigma+\mathcal LV+\theta_0V_y=F,
 \qquad
 \|F(\sigma)\|_{1/5}\le C e^{-c'e^{\rho\sigma}}.
\]
The drift convolution is controlled by
\begin{equation}\label{front:drift-absorption}
 \sup_{\sigma\ge\sigma_0}\int_{\sigma_0}^{\sigma}
 \bigl[1+(\sigma-s)^{-1/2}\bigr]d(s)\,ds
 \le C_\rho d(\sigma_0).
\end{equation}
Indeed, on $\sigma-s\ge1$ use the integrable tail of $d$, and on $0<\sigma-s<1$ use $d(s)\le d(\sigma_0)$ and the integrability of $(\sigma-s)^{-1/2}$. Since $|\theta_0|\le C_\rho d$, \eqref{front:kernel-bounded}--\eqref{front:kernel-derivative} give, for every finite final time $S$,
\[
 \sup_{\sigma_0\le\sigma\le S}\|V(\sigma)\|_{1/5}
 \le C\|V(\sigma_0)\|_{1/5}
 +C\int_{\sigma_0}^S\|F(s)\|_{1/5}\,ds
 +C e^{-a_0\sigma_0}
       \sup_{\sigma_0\le\sigma\le S}\|V(\sigma)\|_{1/5}.
\]
Taking $\sigma_0$ sufficiently large in terms of $\rho$, we absorb the last term and obtain the Gaussian bound. On $0\le y\le2$, the drift and the lifted forcing are uniformly
bounded. Local parabolic gradient estimates up to the flat boundary, applied to $V$ with homogeneous Dirichlet data on unit time intervals, give the second estimate in \eqref{front:linear-upper}; see~\cite{Lieberman1996} for the standard boundary regularity theory. The bound on $h'$ controls the lifted forcing, and adding back $h\chi$ preserves the gradient estimate.
Consequently, $Q(\sigma,2d)\le h+C(2d)\le Cd$.

The first moment in the second part satisfies the exact identity
\begin{equation}\label{front:moment-identity}
 J'=b+\vartheta\int_0^\infty P\,dy
       -\int_0^\infty yAP\,dy.
\end{equation}
The identity can also be obtained from the mild formula using the kernel moment identities above and integration by parts for the drift term. The Gaussian bound justifies the resulting integrals. Thus $|J'|\le Cd$, and integration of the integrable tail gives
\eqref{front:linear-moment}. To prove the profile estimate, separate the
projection \eqref{front:projection} in Duhamel's formula.
The stable part of the initial term is controlled by \eqref{front:kernel-gap}.
For the drift term, use \eqref{front:kernel-derivative} and \eqref{front:kernel-derivative-gap}, interpreting the derivative in the distributional sense.
For the potential term, use
\eqref{front:kernel-bounded}--\eqref{front:kernel-gap}.
Split the boundary term into the most recent unit interval and the preceding
times, and apply the two estimates in \eqref{front:poisson-estimates}.
For $\sigma\ge\sigma_0+1$, the stable terms are bounded in sum by
\[
 C e^{-(\sigma-\sigma_0)}
 +C\int_{\sigma_0}^{\sigma}
 e^{-(\sigma-s)}d(s)\,ds
 +C\int_{\sigma-1}^{\sigma}
 [1+(\sigma-s)^{-1/2}]d(s)\,ds
 \le Cd(\sigma).
\]
Here we use $a_0<1$; the first unit interval is controlled by the assumed weighted bound on $P$, with a constant depending on $\sigma_0$. Hence
$\|P-\Pi P\|_{1/6}\le Cd$, and the first-moment estimate yields
\eqref{front:linear-profile}.

\end{proof}

\subsection{Application to the nonlinear solution}
We now introduce the leading-edge variables that connect the nonlinear transmission problem to the preceding linear estimates. Here $w_\beta$ denotes a rescaled solution, distinct from the parameter variation $w$ in Section~\ref{sec:small-beta}. Recall $m(t)=2t-\frac32\log t$. For $t>0$ and $x=m(t)+z>0$, define
\[
 v_\beta(t,z)=e^zu^\beta(t,m(t)+z),\qquad
 \sigma=\log t,\quad \eta=\frac z{\sqrt t},\qquad
 w_\beta(\sigma,\eta)=t^{-1/2}v_\beta(t,\sqrt t\eta).
\]
A direct calculation gives
\begin{equation}\label{front:selfsimilar-equation}
 (w_\beta)_\sigma+\mathcal Lw_\beta
 +\frac32e^{-\sigma/2}(w_\beta)_\eta
 +t\frac{g(u^\beta)}{u^\beta}w_\beta=0.
\end{equation}
Here $t=e^\sigma$, and $u^\beta$ in the coefficient is evaluated at $x=m(t)+\sqrt t\,\eta$. The equation holds on the moving region $\eta>-m(t)/\sqrt t$, and the coefficient is defined to be zero wherever $u^\beta=0$. Below we work at sufficiently large times on truncated regions contained in $x>0$.

\begin{lemma}\label{lem:leading-edge}
Under Assumptions~\ref{hyp:f}--\ref{hyp:data}, for every finite $\beta\ge0$ there exists a leading-edge coefficient $\alpha_\beta>0$, independent of the auxiliary cutoff exponent $\rho$. Section~\ref{sec:phase} identifies it with $A_*e^{\xi_\beta}$ in Theorem~\ref{thm:phase}.
Fix $0<\rho<1/3$. There are constants $C,T$, depending on $f,u_0,\rho$ but independent of $\beta$, such that $0<\alpha_\beta\le C$ and, for $t\ge T$,
\begin{equation}\label{front:nonlinear-profile}
 \left|w_\beta(\log t,\eta)-\alpha_\beta\eta e^{-\eta^2/4}\right|
       \le Ct^{\rho-1/2}e^{-\eta^2/7},
       \qquad\eta\ge t^{\rho-1/2}.
\end{equation}
In particular, locally uniformly for $\eta\in(0,\infty)$,
\[
 t^{-1/2}e^{\sqrt t\eta}
 u^\beta(t,m(t)+\sqrt t\eta)
       \longrightarrow\alpha_\beta\eta e^{-\eta^2/4}.
\]
If $0<\gamma<1/3$, then
\begin{equation}\label{front:intermediate}
 \frac{e^{t^\gamma}u^\beta(t,m(t)+t^\gamma)}{t^\gamma}
       \longrightarrow\alpha_\beta,
 \qquad
 \frac{e^{t^\gamma}u_x^\beta(t,m(t)+t^\gamma)}{t^\gamma}
       \longrightarrow-\alpha_\beta.
\end{equation}
Moreover, with $\ell(t)=m(t)+t^\rho$,
\begin{equation}\label{front:uniform-moment}
 \mathcal J_\beta(t):=e^{-2t}\int_{\ell(t)}^\infty
 (x-\ell(t))e^xu^\beta(t,x)\,dx,
 \qquad
 |\mathcal J_\beta(t)-2\sqrt\pi\alpha_\beta|
       \le Ct^{\rho-1/2}.
\end{equation}
\end{lemma}

\begin{proof}
Fix $\rho$, and write $d(\sigma)=e^{-(1/2-\rho)\sigma}$. By \eqref{eq:uniform-gaussian}, at any fixed positive time the weighted right tail has a Gaussian bound uniform in $\beta$. Choose a common starting time $\sigma_*$ sufficiently large for Lemma~\ref{lem:linear-front} and for the negative cutoff $z=-t^\rho$ to lie in $x>0$ for all subsequent times. Write $t_*=e^{\sigma_*}$. This cutoff uses the bound $u^\beta\le1$ to make the boundary input small:
\[
 0\le w_\beta(\sigma,-d(\sigma))
       \le h(\sigma):=e^{-\sigma/2-e^{\rho\sigma}}.
\]
Remove the nonnegative absorption term from \eqref{front:selfsimilar-equation} and construct a common linear supersolution with boundary value $h$. Its initial profile can be chosen continuous, nonnegative, dominating every $w_\beta(\sigma_*,\cdot)$ on the truncated region, with a common $X_{1/5}$ bound after shifting the boundary to zero, and with value $h(\sigma_*)$ at that boundary. Indeed, take the minimum of the bound $t_*^{-1/2}e^{\sqrt{t_*}\eta}$ and a Gaussian upper envelope furnished by \eqref{eq:uniform-gaussian}, enlarged to be at least $h(\sigma_*)$ at the cutoff. This also ensures initial-boundary compatibility.
Move the negative cutoff to zero by setting $y=\eta+d(\sigma)$.
The drift coefficient is exactly
$\theta_0=\rho d+\frac32e^{-\sigma/2}$, so the first part of
Lemma~\ref{lem:linear-front} applies; $h$ and $h'$ satisfy its required decay bound.
For $\sigma\ge\sigma_*+1$, the gradient bound for the linear supersolution and comparison yield
\begin{equation}\label{front:positive-cutoff-bound}
 w_\beta(\sigma,d(\sigma))\le Cd(\sigma).
\end{equation}
The point $\eta=d$ corresponds to $y=2d$ in the supersolution coordinates. We increase the threshold time $T$ to include this one-unit waiting interval.

The positive cutoff makes the nonlinear potential uniformly small. Move this cutoff to zero and set
\[
 \mathcal P_\beta(\sigma,y)=w_\beta(\sigma,y+d(\sigma)).
\]
The coefficients in \eqref{front:linear-positive-cut} are now
\begin{equation}\label{front:positive-coefficients}
 \vartheta=-\rho d+\frac32e^{-\sigma/2},\qquad
 A_\beta(\sigma,y)=t\frac{g(u^\beta(t,x))}{u^\beta(t,x)}\ge0,
 \qquad x=m(t)+\sqrt t\,(y+d).
\end{equation}
The Gaussian bound for the common supersolution and \eqref{eq:g} give,
on $z\ge t^\rho$,
\[
 0\le A_\beta\le Ctu^\beta
       \le Ct^{3/2}e^{-t^\rho}.
\]
After the shift, the Gaussian bound remains uniform with the weaker weight
$1/6$. Together with \eqref{front:positive-cutoff-bound}, the second part of
Lemma~\ref{lem:linear-front}, started at $\sigma_*+1$, yields a nonnegative amplitude and the profile and first-moment estimates. Positive-time regularity gives continuous coefficients, boundary traces, and compatible initial values at this starting time. The common weighted bound also yields a uniform upper bound for the amplitudes. Returning from $y=\eta-d$ to $\eta$ and weakening the Gaussian
weight slightly once more gives \eqref{front:nonlinear-profile}.
The change of variables also gives the exact identity
\[
 \int_0^\infty y\mathcal P_\beta(\log t,y)\,dy=\mathcal J_\beta(t),
\]
so \eqref{front:uniform-moment} holds with the same common constant. If two different cutoff exponents are used, the resulting profile limits at any fixed $\eta>0$ must agree. Since $\phi(\eta)>0$, their amplitudes coincide; we denote this common value by $\alpha_\beta$.

We now prove strict positivity for each fixed $\beta$. First fix a nonzero function $\psi\in C_c^\infty((1,2))$ with $\psi\ge0$. Choose a time $\sigma_1\ge\sigma_*+1$ sufficiently large that the decaying drift and potential can be absorbed in the homogeneous-boundary Duhamel estimate below and that
\[
 C\|\psi\|_{1/6}\int_{\sigma_1}^\infty
 (|\vartheta|+\|A_\beta\|_\infty)\,d\sigma
 <\frac12\int_0^\infty y\psi(y)\,dy,
\]
where $C$ is the constant in that estimate, independent of sufficiently large $\sigma_1$. Finally, positivity of the transmission kernel and comparison with the linear heat evolution give $\mathcal P_\beta(\sigma_1,y)>0$ for $y>0$. Hence one can choose $\varepsilon>0$ such that $\varepsilon\psi\le \mathcal P_\beta(\sigma_1,\cdot)$.

Let $H$ solve \eqref{front:linear-positive-cut} with the actual coefficients $\vartheta,A_\beta$, homogeneous Dirichlet boundary data, and initial datum $\varepsilon\psi$ at $\sigma_1$. Comparison gives $0\le H\le \mathcal P_\beta$. In Duhamel's formula, \eqref{front:drift-absorption} bounds the drift contribution by $Cd(\sigma_1)$ times the weighted supremum norm. The potential contribution is bounded by $C\int_{\sigma_1}^\infty\|A_\beta(s)\|_\infty\,ds$ times that norm, using \eqref{front:kernel-bounded}. Both coefficients are small for sufficiently large $\sigma_1$. Absorbing these terms gives
\[
 \sup_{\sigma\ge\sigma_1}\|H(\sigma)\|_{1/6}
       \le C\varepsilon\|\psi\|_{1/6}.
\]
By \eqref{front:moment-identity} with zero boundary data, the total variation of its first moment on $[\sigma_1,\infty)$ is at most
\[
 C\varepsilon\|\psi\|_{1/6}
 \int_{\sigma_1}^\infty
 (|\vartheta|+\|A_\beta\|_\infty)\,d\sigma
 <\frac\varepsilon2\int_0^\infty y\psi(y)\,dy.
\]
Thus the limiting first moment of $H$ is positive. Since $H\le \mathcal P_\beta$, this implies $\alpha_\beta>0$. The seed size $\varepsilon$ may depend on $\beta$, so this argument does not assert a positive lower bound uniform over all parameters.

For the intermediate-scale limits, given $\gamma$, choose $0<\rho<\gamma$.
At $\eta=t^{\gamma-1/2}$, the profile error divided by $\eta$ is
$O(t^{\rho-\gamma})$, which proves the first limit in
\eqref{front:intermediate}. For the gradient limit, use the unshifted self-similar coordinates and set
\[
 E_\beta=w_\beta-\alpha_\beta\phi,\qquad
 a(\sigma)=\frac32e^{-\sigma/2},\qquad
 \widetilde A_\beta(\sigma,\eta)
 =t\frac{g(u^\beta(t,m(t)+\sqrt t\eta))}
          {u^\beta(t,m(t)+\sqrt t\eta)}.
\]
Then
\[
 (E_\beta)_\sigma+\mathcal LE_\beta
 +a(E_\beta)_\eta+\widetilde A_\beta E_\beta
 =-\alpha_\beta\bigl(a\phi'+\widetilde A_\beta\phi\bigr).
\]
For $2d(\sigma)\le\eta\le1$ and sufficiently large $\sigma$, take a backward parabolic cylinder with spatial radius $\eta/4$ and time length $(\eta/4)^2$. It remains to the right of the moving positive cutoff, and $d$ is comparable to $d(\sigma)$ throughout it. On this cylinder, \eqref{front:nonlinear-profile} bounds $E_\beta$ by $Cd$, while the potential is superexponentially small and the amplitudes are uniformly bounded. The scaled interior parabolic gradient estimate (see~\cite{Lieberman1996}) therefore gives
\[
 |\partial_\eta E_\beta(\sigma,\eta)|
 \le C\left(\frac d\eta
       +\eta\bigl(e^{-\sigma/2}
          +\|\widetilde A_\beta\|_{L^\infty(\text{cylinder})}\bigr)\right)
 \le C\frac d\eta.
\]
At the intermediate scale above, this yields
$\partial_zv_\beta=\alpha_\beta+o(1)$.
Using $u_x^\beta=e^{-z}(\partial_zv_\beta-v_\beta)$ and dividing by
$z=t^\gamma$ proves the second limit in \eqref{front:intermediate}.
\end{proof}

\subsection{Application to signed difference quotients}
The next corollary isolates the common bounds needed to control a family of signed solutions. Its application to parameter differences will require separate estimates near the interface.

\begin{corollary}\label{cor:signed-amplitudes}
Under Assumptions~\ref{hyp:f}--\ref{hyp:data}, fix $0<\rho<1/3$, and set
\[
 \ell(t)=m(t)+t^\rho,\qquad d(\sigma)=e^{-(1/2-\rho)\sigma}.
\]
Choose a common $T_0>1$ sufficiently large that $m(t)-t^\rho>0$ for $t\ge T_0$. Consider a family of classical solutions $W$ on $t>T_0$, $x>0$, continuous up to $t=T_0$ and bounded on each finite time strip, satisfying
\[
 W_t=W_{xx}+c(t,x)W,
\]
where $c$ is continuous, locally bounded, and $c\le1$. In the negative-cutoff coordinates, define
\[
 Z_W(\sigma,y)=t^{-1/2}e^{\sqrt t\,(y-d(\sigma))}
 W\!\left(t,m(t)+\sqrt t\,(y-d(\sigma))\right),
 \qquad t=e^\sigma,\quad y\ge0.
\]
Suppose that common constants $M,c_0>0$ give
\[
 |Z_W(\sigma,0)|\le M e^{-c_0e^{\rho\sigma}}
 \quad(\sigma\ge\log T_0),\qquad
 \|Z_W(\log T_0,\cdot)\|_{1/5}\le M.
\]
For each member of the family, suppose also that there are finite $\beta_1,\beta_2\ge0$ such that
\[
 |1-c(t,x)|\le M\bigl(u^{\beta_1}(t,x)+u^{\beta_2}(t,x)\bigr),
 \qquad t\ge T_0,\quad x\ge\ell(t).
\]
Then each $W$ has a real amplitude $\alpha_W$, and there are common constants $C,T$ such that $|\alpha_W|\le C$ and, for $t\ge T$,
\begin{equation}\label{front:signed-moment}
 \left|e^{-2t}\int_{\ell(t)}^\infty
 (x-\ell(t))e^xW(t,x)\,dx-2\sqrt\pi\alpha_W\right|
 \le Ct^{\rho-1/2}.
\end{equation}
Moreover, defining
\[
 P_W(\log t,y)=t e^{x-2t}W(t,x),
 \qquad x=\ell(t)+\sqrt t\,y,
\]
we have $\|P_W(\log t,\cdot)-\alpha_W\phi\|_{1/6}\le Ct^{\rho-1/2}$. For each $\rho<\gamma<1/3$,
\begin{equation}\label{front:signed-intermediate}
 \frac{e^{t^\gamma}W(t,m(t)+t^\gamma)}{t^\gamma}
 \longrightarrow\alpha_W,
\end{equation}
uniformly over the family. The constants depend only on $f,u_0,\rho,T_0,M,c_0$, and not on the individual solution or the parameters $\beta_1,\beta_2$.
\end{corollary}

\begin{proof}
To obtain Kato's inequality, use the smooth convex approximation
$\Phi_\delta(s)=\sqrt{s^2+\delta^2}$. The equation for $W$ gives
\[
 (\partial_t-\partial_{xx})\Phi_\delta(W)
 =cW\Phi_\delta'(W)-\Phi_\delta''(W)|W_x|^2
 \le cW\Phi_\delta'(W).
\]
Letting $\delta\downarrow0$ against nonnegative compactly supported test
functions yields
$\partial_t|W|-\partial_{xx}|W|\le c|W|\le|W|$ in distributions.
The same argument for an additional bounded source $F$ gives the upper
bound $c|W|+|F|$, a form used in Section~\ref{sec:small-beta}.
At the negative cutoff, prescribe the smooth upper envelope
\[
 h(\sigma)=M_1\exp(-c_0e^{\rho\sigma}),\qquad
 |h|+|h'|\le C\exp\bigl(-(c_0/2)e^{\rho\sigma}\bigr).
\]
Choose $M_1$ large enough, in terms of the common data, that $h(\log T_0)\ge M$. Then $h(\log T_0)e^{-y^2/5}$ is a common dominating initial profile in the shifted negative-cutoff coordinates, compatible with $h$ at $y=0$. The linear supersolution construction used in the preceding subsection applies to $|W|$. If necessary, propagate this common supersolution to a later fixed time before applying Lemma~\ref{lem:linear-front}(i). On this finite interval, the kernel estimates and a Volterra Gronwall estimate with the integrable singularity $(\sigma-s)^{-1/2}$ preserve a common $X_{1/5}$ bound. Thus the bound at the later starting time depends only on the listed common data. After a further one-unit waiting interval, its gradient bound over the distance $2d$ between the cutoffs gives
\[
 |P_W(\sigma,0)|\le Cd(\sigma),\qquad
 \|P_W(\sigma,\cdot)\|_{1/6}\le C.
\]
No derivative estimate for the actual boundary trace of $W$ is needed.

In the positive-cutoff coordinates, direct calculation gives
\[
 (P_W)_\sigma+\mathcal LP_W
 +\left(-\rho d+\frac32e^{-\sigma/2}\right)(P_W)_y
 +A_WP_W=0,
 \qquad A_W(\sigma,y)=t(1-c(t,x)).
\]
The coefficient assumption and the preceding nonlinear estimates imply
\[
 0\le A_W\le Mt\bigl(u^{\beta_1}+u^{\beta_2}\bigr)
 \le Ct^{3/2}e^{-t^\rho},\qquad x\ge\ell(t).
\]
Thus all bounds in Lemma~\ref{lem:linear-front}(ii) hold with common constants, and positive-time continuity gives compatible initial and boundary traces at the chosen starting time. That lemma yields the asserted profile estimate and a bounded family of real amplitudes. The identity
\[
 \int_0^\infty yP_W(\log t,y)\,dy
 =e^{-2t}\int_{\ell(t)}^\infty(x-\ell(t))e^xW(t,x)\,dx
\]
gives \eqref{front:signed-moment}. Finally, for $\gamma>\rho$, set $y=t^{\gamma-1/2}-d$. The profile error, after division by $t^{\gamma-1/2}$, is $O(t^{\rho-\gamma})$, while $y/t^{\gamma-1/2}\to1$ and $\phi(y)/y\to1$. This proves \eqref{front:signed-intermediate} uniformly.
\end{proof}

If the hypotheses hold for every sufficiently small $\rho>0$, the amplitudes obtained with different cutoffs agree by evaluating the profile limit at a fixed positive self-similar coordinate. One may then choose $\rho<\gamma$ to obtain \eqref{front:signed-intermediate} for every $0<\gamma<1/3$; the common bounds may depend on that choice of $\rho$.

Corollary~\ref{cor:signed-amplitudes} will be applied to the weak-barrier
difference quotient in Section~\ref{sec:small-beta}, where we establish
the common boundary and initial weighted bounds and verify the coefficient
hypotheses. Combined with finite-time differentiability, its uniform moment
estimate will identify the derivative of the amplitude. We first use the
positive nonlinear amplitude to determine the wave phase for each fixed
finite resistance.

\section{Phase and far-field time delay for a fixed resistance}
\label{sec:phase}

Throughout this section, Assumptions~\ref{hyp:f}--\ref{hyp:data} hold and the resistance $0\le\beta<\infty$ is fixed. We prove the phase and arrival-time conclusions of Theorem~\ref{thm:phase} without assuming concavity of $f$. The corresponding homogeneous convergence results are discussed in~\cite{HNRR2013,NRR2017}.

Lemma~\ref{lem:leading-edge} provides a positive leading-edge amplitude $\alpha_\beta$. Since the translated wave has the tail
\[
 U_*(z-\xi)\sim A_*e^\xi z e^{-z}\qquad(z\to+\infty),
\]
matching its leading coefficient with $\alpha_\beta$ determines the candidate phase $\xi_\beta=\log(\alpha_\beta/A_*)$. We will prove that this candidate indeed gives convergence on the entire receiving side:
\begin{equation}\label{phase:uniform-wave}
 \lim_{t\to\infty}\sup_{x\ge0}
 \left|u^\beta(t,x)-U_*(x-m(t)-\xi_\beta)\right|=0,
 \qquad
 \xi_\beta=\log\frac{\alpha_\beta}{A_*}.
\end{equation}
As before, the value at $x=0$ is the right-hand trace. For the first-arrival times defined in \eqref{eq:arrival}, we will then show that $T_\eta^\beta(R)<\infty$ for every $R>0$ and $\eta\in(0,1)$, and that
\begin{equation}\label{phase:delay-formula}
 \lim_{R\to\infty}\left[T_\eta^\beta(R)-T_\eta^0(R)\right]
 =\frac{\xi_0-\xi_\beta}{2}
 =\frac12\log\frac{\alpha_0}{\alpha_\beta}.
\end{equation}
In particular, the limiting delay is independent of the arrival level.

After recalling the spatial monotonicity and finite-time tail estimates,
the proof proceeds in three steps. We first match the leading-edge amplitude to the interior wave profile on a moving interval lying to the right of the interface. Spatial monotonicity then extends convergence on bounded intervals in the moving frame to the entire receiving half-line. Finally, a finite-time tail bound excludes early arrivals at distant points, allowing the phase convergence to yield \eqref{phase:delay-formula} without temporal monotonicity.

\subsection{Spatial monotonicity and finite-time tail bounds}

We recall two consequences of Section~\ref{sec:preliminaries}. By \eqref{eq:right-monotone}, $u^\beta(t,\cdot)$ is strictly decreasing on the receiving side for every $t>0$. Together with $0\le u^\beta\le1$, this will control both tails once convergence near the front has been established.

For every finite $T>0$, \eqref{eq:uniform-gaussian} gives
\begin{equation}\label{phase:finite-time-tail}
 \sup_{0<t\le T}u^\beta(t,x)
 \le e^T\operatorname{erfc}\!\left(\frac{x}{2\sqrt T}\right),\qquad x>0.
\end{equation}
Indeed, both $e^t$ and $\erfc(x/(2\sqrt t))$ are increasing in $t$ for fixed $x>0$. The bound is independent of the finite resistance, and hence
\[
 \lim_{x\to+\infty}\sup_{0\le\beta<\infty}\sup_{0<t\le T}
 u^\beta(t,x)=0.
\]
Thus, for each arrival level $\eta\in(0,1)$ and each fixed $T$, sufficiently distant points cannot reach that level during $0<t\le T$, uniformly in $\beta$. In particular, the first-arrival times tend to infinity as $R\to\infty$, with $+\infty$ still allowed at this stage. Their finiteness will follow from phase convergence.

These two properties will be used after we obtain local convergence in the moving frame. We next establish the interior matching result needed for that step.

\subsection{Renormalized energy on a moving interval}

Fix $0<\gamma<1/3$ and set $R(t)=t^\gamma$.
The right boundary lies in the leading edge and prescribes its amplitude. The next lemma shows that this amplitude selects the interior wave profile, while the left boundary data contribute only an integrable error after exponential weighting. It therefore allows the trace of the actual solution at the left boundary to be used without assuming convergence behind the front in advance.

\begin{lemma}\label{matching:energy-lemma}
Suppose that Assumption~\ref{hyp:f} holds. Given $\alpha>0$ and $T\ge1$, let $h\in C^1([T,\infty))$ satisfy
$0\le h\le1$ and $\sup_{t\ge T}|h'(t)|<\infty$.
Let $q_T$ be a measurable function on $(-R(T),R(T))$ with $0\le q_T\le1$. On $-R(t)<z<R(t)$, let $q_{\mathrm{aux}}$ be the bounded solution of
\begin{equation}\label{matching:auxiliary-equation}
 \begin{cases}
 (q_{\mathrm{aux}})_t=(q_{\mathrm{aux}})_{zz}
  +\left(2-\dfrac3{2t}\right)(q_{\mathrm{aux}})_z
  +f(q_{\mathrm{aux}}),\\[1mm]
 q_{\mathrm{aux}}(t,-R(t))=h(t),\\
 q_{\mathrm{aux}}(t,R(t))=\alpha R(t)e^{-R(t)},\\
 q_{\mathrm{aux}}(T,z)=q_T(z).
 \end{cases}
\end{equation}
The equation and boundary conditions hold for $t>T$. The solution is understood in the mild sense, is classical in the interior for $t>T$, and attains the lateral boundary data continuously away from the initial corners. Its initial trace is $q_T$ locally in $L^p((-R(T),R(T)))$ for $1\le p<\infty$; no compatibility at the initial corners is required. Assume that $T$ is large enough that $\alpha R(t)e^{-R(t)}\le1$ for all $t\ge T$. Then, for every $0<K<\infty$,
\begin{equation}\label{matching:auxiliary-limit}
 \lim_{t\to\infty}\sup_{|z|\le K}
 \left|q_{\mathrm{aux}}(t,z)
 -U_*\!\left(z-\log\frac\alpha{A_*}\right)\right|=0.
\end{equation}
\end{lemma}

\begin{proof}
Write $q=q_{\mathrm{aux}}$. Standard parabolic theory, after mapping the moving interval to a fixed interval, gives the bounded solution and positive-time regularity; see~\cite{Lieberman1996}. Comparison gives $0\le q\le1$. All energy calculations below start at $T+1$, so initial corner incompatibilities cause no difficulty. The identities may first be derived for smooth data and then justified by approximation on time intervals separated from $T$.
Recall $g(s)=s-f(s)$ from~\eqref{eq:g}, and set
\[
 G(s)=\int_0^s g(r)\,dr,\qquad V(t,z)=e^zq(t,z).
\]
By \eqref{eq:g}, $0\le g(s)\le Cs^2$, hence $0\le G(s)\le Cs^3$.
The function $V$ satisfies
\begin{equation}\label{matching:weighted-equation}
 V_t-V_{zz}+a(t)(V_z-V)+P_V(z,V)=0,
 \qquad a(t)=\frac3{2t},
\end{equation}
where
\[
 P(z,V)=e^{2z}G(e^{-z}V),\qquad
 P_V(z,V)=e^zg(e^{-z}V)\ge0.
\]
The right boundary condition becomes $V(t,R(t))=\alpha R(t)$, corresponding to the linear leading-edge growth $V\sim\alpha z$. Its gradient energy on the positive half of the interval has the leading contribution $\alpha^2R(t)/2$. Subtracting this contribution, define
\begin{equation}\label{matching:renormalized-energy}
 E(t)=\int_{-R(t)}^{R(t)}
 \left(\frac12V_z^2+P(z,V)\right)\,dz
 -\frac{\alpha^2}2R(t).
\end{equation}
Using that boundary condition and completing the square gives
\begin{align}\label{matching:coercive-identity}
 E(t)+\alpha
 ={}&\frac12\int_{-R(t)}^0V_z^2\,dz
 +\frac12\int_0^{R(t)}(V_z-\alpha)^2\,dz \notag\\
 &+\int_{-R(t)}^{R(t)}P(z,V)\,dz
 +\alpha[1-V(t,0)].
\end{align}
Since $V(t,0)=q(t,0)\le1$, every term on the right is nonnegative.
In particular, $E\ge-\alpha$.

The energy also controls the norm needed to estimate the drift term.
On the negative half of the interval we use $V\le e^z$; on the positive half we use
\[
 V(t,z)=V(t,0)+\alpha z
 +\int_0^z[V_z(t,s)-\alpha]\,ds.
\]
The Cauchy--Schwarz inequality and~\eqref{matching:coercive-identity} give
\begin{equation}\label{matching:drift-control}
 \int_{-R(t)}^{R(t)}(V_z-V)^2\,dz
 \le CR(t)^2[E(t)+\alpha+1]+CR(t)^3.
\end{equation}
The constants below may depend on $\alpha,T,\gamma,f$ and $\|h'\|_{L^\infty([T,\infty))}$, but not on $t$. The resulting energy and dissipation bounds may additionally depend on the initial profile through the finite energy $E(T+1)$.

Differentiate~\eqref{matching:renormalized-energy} and use
\eqref{matching:weighted-equation} to eliminate $V_{zz}-P_V$.
The interior contribution is
\[
 -\int_{-R}^R V_t^2\,dz
 -a(t)\int_{-R}^R(V_z-V)V_t\,dz.
\]
At the moving right endpoint, $V_t+R'V_z=\alpha R'$.
Thus the right boundary contribution, including the renormalization term, is exactly
\begin{equation}\label{matching:right-boundary-flux}
 -\frac{R'}2[V_z(t,R)-\alpha]^2
 +R'P(R,\alpha R).
\end{equation}
Since $P(R,\alpha R)\le C\alpha^3R^3e^{-R}$, its positive part is integrable.
At the left endpoint, write $b(t)=e^{-R(t)}h(t)$. Differentiation along this moving endpoint gives
\[
 V_t(t,-R)-R'V_z(t,-R)=b'(t).
\]
The corresponding boundary contribution is
\begin{equation}\label{matching:left-boundary-flux}
 -V_z(t,-R)b'(t)-\frac{R'}2V_z(t,-R)^2
 +R'P(-R,b).
\end{equation}
Since $R'(t)>0$, Young's inequality gives directly
\[
 -V_z(t,-R)b'-\frac{R'}2V_z(t,-R)^2
 \le\frac{|b'|^2}{2R'}.
\]
Here $b'=e^{-R}(h'-R'h)$ and $P(-R,b)=e^{-2R}G(h)$. Hence the positive part of \eqref{matching:left-boundary-flux} is bounded by
\[
 \frac{|b'|^2}{2R'}+R'P(-R,b)
 \le C\bigl(t^{1-\gamma}+t^{\gamma-1}\bigr)e^{-2t^\gamma},
\]
which is integrable. This polynomial factor can be absorbed into a weaker exponential bound $Ce^{-R(t)}$.

Use Young's inequality to absorb the drift term into the dissipation.
Together with~\eqref{matching:drift-control}, this gives
\begin{align}\label{matching:energy-dissipation}
 E'(t)\le{}&-\frac12\int_{-R(t)}^{R(t)}V_t^2\,dz
 +Ct^{2\gamma-2}[E(t)+\alpha+1] \notag\\
 &+Ct^{3\gamma-2}+C(1+R(t)^3)e^{-R(t)}.
\end{align}
Both $t^{2\gamma-2}$ and $t^{3\gamma-2}$ are integrable, the latter being the term that requires $\gamma<1/3$; the exponential remainder is integrable as well. Apply Gronwall's inequality to $E+\alpha+1\ge1$,
and then integrate~\eqref{matching:energy-dissipation}, to obtain
\begin{equation}\label{matching:finite-dissipation}
 \sup_{t\ge T+1}E(t)<\infty,\qquad
 \int_{T+1}^\infty\int_{-R(t)}^{R(t)}V_t^2\,dz\,dt<\infty.
\end{equation}

We next identify the limits of time translates. Take any sequence $t_n\to\infty$. On fixed compact subsets of space-time, $0\le q\le1$ and interior parabolic estimates give local compactness of $V(t_n+s,z)$. After passing to a subsequence, these functions and their first spatial derivatives converge uniformly on compact subsets. For fixed $K,S>0$, the expanding domain eventually contains $[-K,K]$ throughout $[t_n-S,t_n+S]$, and \eqref{matching:finite-dissipation} yields
\[
 \int_{t_n-S}^{t_n+S}\int_{-K}^K|V_t(t,z)|^2\,dz\,dt
 \longrightarrow0.
\]
Thus every subsequential limit is independent of $s$. Passing to the limit in the equation in the distributional sense, using $a(t_n+s)\to0$ uniformly for bounded $s$, and then applying local regularity, we find that its stationary profile $V_\infty$ satisfies
\begin{equation}\label{matching:stationary-limit}
 V_\infty''=e^zg(e^{-z}V_\infty),\qquad
 0\le e^{-z}V_\infty\le1.
\end{equation}
By~\eqref{matching:coercive-identity} and the uniform energy bound, we may first pass to the limit on $[0,K]$ and then let $K\to\infty$ to obtain
\begin{equation}\label{matching:tail-slope-control}
 \int_0^\infty[V_\infty'(z)-\alpha]^2\,dz<\infty.
\end{equation}
Equation~\eqref{matching:stationary-limit} shows that $V_\infty'$ is nondecreasing.
Consequently,~\eqref{matching:tail-slope-control} forces
\begin{equation}\label{matching:identified-slope}
 \lim_{z\to\infty}V_\infty'(z)=\alpha,
 \qquad V_\infty(z)=\alpha z+o(z).
\end{equation}
In particular, $V_\infty\not\equiv0$.

It remains to identify this stationary profile as a translated wave. Set $Q=e^{-z}V_\infty$. Then
\begin{equation}\label{matching:wave-ode}
 Q''+2Q'+f(Q)=0,\qquad 0\le Q\le1,\qquad Q(+\infty)=0.
\end{equation}
To identify the limit at the other end, set $F(s)=\int_0^s f(r)\,dr$.
Local ODE estimates give a uniform bound on the derivative of the bounded
solution $Q$, whose mechanical energy satisfies
\[
 \frac d{dz}\left(\frac12Q'^2+F(Q)\right)=-2Q'^2.
\]
The energy has finite limits at both ends, so $Q'\in L^2(\mathbb R)$.
The boundedness of $Q''=-2Q'-f(Q)$ makes $Q'$ uniformly continuous,
and therefore $Q'(\pm\infty)=0$. The boundedness of
$Q'''=-2Q''-f'(Q)Q'$ also makes $Q''$ uniformly continuous. Together with the finite limits of $Q'$ at both ends, this implies $Q''(\pm\infty)=0$: otherwise, uniform continuity would force changes of a fixed size in $Q'$ along intervals escaping to one of the two ends. The ODE therefore gives $f(Q(z))\to0$. For every $0<\delta<1/2$, positivity of $f$ on $[\delta,1-\delta]$ implies that $Q$ eventually avoids this interval at either end. Continuity prevents repeated transitions between the neighborhoods of $0$ and $1$ on either tail. Letting $\delta\downarrow0$ shows that each end limit exists and is either $0$ or $1$. If $Q(-\infty)=0$, both limiting
energies vanish, and the energy identity forces $Q'\equiv0$, contradicting
\eqref{matching:identified-slope}. Thus $Q(-\infty)=1$. The strong maximum
principle gives $0<Q<1$. At any point where $Q'=0$, one has
$Q''=-f(Q)<0$, so every critical point would be a strict local maximum. If $Q'$ were positive somewhere, it would remain positive to the left, since a transition from negative to positive would require a local minimum. This contradicts $Q(-\infty)=1$ and $Q<1$. Thus $Q'\le0$, and a zero of $Q'$ would itself force positive values immediately to its left. Hence $Q'<0$.
By the classical uniqueness recalled with~\eqref{eq:wave}, $Q(z)=U_*(z-\xi)$.

The expansion of the wave tail gives
\[
 \lim_{z\to\infty}\frac{e^zU_*(z-\xi)}z=A_*e^\xi.
\]
Comparison with~\eqref{matching:identified-slope} yields
$\xi=\log(\alpha/A_*)$. Every subsequence of time translates has the
same stationary limit. Hence the full family converges locally uniformly,
as asserted in~\eqref{matching:auxiliary-limit}.
\end{proof}

\subsection{Comparison with the actual solution and uniqueness of the phase}

We apply the matching lemma with the actual initial profile and left trace, and with two nearby right-hand amplitudes.

\begin{proof}[Proof of the phase statement in Theorem~\ref{thm:phase}]
Fix a finite $\beta\ge0$, choose $0<\gamma<1/3$, and set $R(t)=t^\gamma$. Given $0<\varepsilon<\alpha_\beta$, Lemma~\ref{lem:leading-edge} allows us to choose $T$ sufficiently large that
\begin{equation}\label{phase:front-amplitude-bracket}
 (\alpha_\beta-\varepsilon)R(t)e^{-R(t)}
 \le u^\beta(t,m(t)+R(t))
 \le(\alpha_\beta+\varepsilon)R(t)e^{-R(t)},
 \qquad t\ge T.
\end{equation}
Increase $T$ if necessary so that the moving interval stays a fixed positive distance from the interface and $(\alpha_\beta+\varepsilon)R(t)e^{-R(t)}\le1$ for all $t\ge T$.

Set $\widehat u_\beta(t,z)=u^\beta(t,m(t)+z)$. Since $m'(t)=2-3/(2t)$, this function satisfies the same equation as the auxiliary solutions in \eqref{matching:auxiliary-equation}. We apply Lemma~\ref{matching:energy-lemma} with amplitudes $\alpha_\beta\pm\varepsilon$, common initial data
\[
 q_{\mathrm{aux}}^\pm(T,z)=\widehat u_\beta(T,z),\qquad |z|<R(T),
\]
and common left boundary data
\[
 h(t)=u^\beta(t,m(t)-R(t)).
\]
These data lie in $[0,1]$. Moreover,
\[
 h'(t)=u_t^\beta(t,m(t)-R(t))
 +(m'(t)-R'(t))u_x^\beta(t,m(t)-R(t)).
\]
Interior parabolic estimates away from the interface bound $u_t^\beta$ and $u_x^\beta$ along this curve, while $m'-R'$ is bounded. Thus $h\in C^1([T,\infty))$ has bounded derivative, with constants allowed to depend on the fixed $\beta$. All hypotheses of the matching lemma are satisfied. Possible incompatibilities at the initial right corners are covered by its formulation.

The two auxiliary problems have the same initial profile and left boundary value as $\widehat u_\beta$, and their right boundary values bracket its trace by \eqref{phase:front-amplitude-bracket}. Comparison therefore gives
\[
 q_{\mathrm{aux}}^-(t,z)
 \le\widehat u_\beta(t,z)
 \le q_{\mathrm{aux}}^+(t,z),\qquad t>T,\quad |z|\le R(t).
\]
For every fixed $K>0$, the matching lemma gives uniform convergence of these two bounds on $[-K,K]$ to
\[
 U_*\!\left(z-\log\frac{\alpha_\beta-\varepsilon}{A_*}\right)
 \quad\hbox{and}\quad
 U_*\!\left(z-\log\frac{\alpha_\beta+\varepsilon}{A_*}\right),
\]
respectively. Both profiles tend uniformly on $[-K,K]$ to $U_*(z-\xi_\beta)$ as $\varepsilon\downarrow0$. Taking $t\to\infty$ first and then $\varepsilon\downarrow0$ in the comparison yields
\begin{equation}\label{phase:local-wave-limit}
 u^\beta(t,m(t)+z)\longrightarrow U_*(z-\xi_\beta),
 \qquad \xi_\beta=\log\frac{\alpha_\beta}{A_*},
\end{equation}
locally uniformly in $z$.

We extend this convergence to the entire receiving half-line using \eqref{eq:right-monotone}. Fix $K>0$ and take $t$ large enough that $m(t)>K$. On $0\le x\le m(t)-K$, their deviations from $1$ are controlled by the values at $x=m(t)-K$, since
\[
 u^\beta(t,m(t)-K)\le u^\beta(t,x)\le1,
 \qquad U_*(-K-\xi_\beta)\le U_*(x-m(t)-\xi_\beta)\le1.
\]
At $x=0$, these inequalities are understood using the right-hand trace. On $x\ge m(t)+K$, their values are controlled by those at $x=m(t)+K$, since
\[
 0\le u^\beta(t,x)\le u^\beta(t,m(t)+K),
 \qquad 0\le U_*(x-m(t)-\xi_\beta)\le U_*(K-\xi_\beta).
\]
On the intervening interval, \eqref{phase:local-wave-limit} gives uniform convergence. Consequently,
\[
 \limsup_{t\to\infty}\sup_{x\ge0}
 |u^\beta(t,x)-U_*(x-m(t)-\xi_\beta)|
 \le\max\{1-U_*(-K-\xi_\beta),\,U_*(K-\xi_\beta)\}.
\]
Letting $K\to\infty$ proves \eqref{phase:uniform-wave}.

Finally, suppose that another phase $\widetilde\xi_\beta$ gives the same convergence. Evaluating both limits at $x=m(t)+z$ for any fixed $z$ yields
\[
 U_*(z-\xi_\beta)=U_*(z-\widetilde\xi_\beta).
\]
Strict monotonicity of $U_*$ implies $\widetilde\xi_\beta=\xi_\beta$, proving uniqueness.
\end{proof}

\subsection{First-arrival bounds without temporal monotonicity}

We combine phase convergence, spatial monotonicity, and the finite-time tail bound to bracket the first-arrival times.

\begin{proof}[Proof of the arrival-time statement in Theorem~\ref{thm:phase}]
Fix $\beta\in[0,\infty)$ and $\eta\in(0,1)$, and set $z_\eta=U_*^{-1}(\eta)$.
For every fixed $R>0$,~\eqref{phase:uniform-wave} and $m(t)\to\infty$ imply
$u^\beta(t,R)\to1$. Thus the set in~\eqref{eq:arrival} is nonempty, and
$T_\eta^\beta(R)<\infty$ for every $R>0$.

The same convergence gives $u^\beta(t,0^+)\to1$, whereas
\eqref{phase:finite-time-tail} implies $u^\beta(t,x)\to0$ as $x\to\infty$
for each fixed $t>0$. Continuity and strict spatial monotonicity therefore
provide a unique positive level-set position $X_\eta^\beta(t)$ for all
sufficiently large $t$. To locate it, fix $\delta>0$. Since
\[
 U_*(z_\eta-\delta)>\eta>U_*(z_\eta+\delta),
\]
\eqref{phase:uniform-wave} gives, for all sufficiently large $t$,
\[
 u^\beta(t,m(t)+\xi_\beta+z_\eta-\delta)>\eta,
 \qquad
 u^\beta(t,m(t)+\xi_\beta+z_\eta+\delta)<\eta.
\]
Both spatial arguments are then positive. Strict spatial monotonicity places
$X_\eta^\beta(t)$ between them. As $\delta>0$ is arbitrary, this proves
\begin{equation}\label{phase:level-location}
 \begin{gathered}
 u^\beta(t,X_\eta^\beta(t))=\eta,\\
 X_\eta^\beta(t)=m(t)+\xi_\beta+z_\eta+o(1).
 \end{gathered}
\end{equation}

For a fixed $\delta>0$, choose $T_\delta$ large enough that $m$ is strictly
increasing on $[T_\delta,\infty)$ and
\[
 |X_\eta^\beta(t)-m(t)-\xi_\beta-z_\eta|<\delta,
 \qquad t\ge T_\delta.
\]
For sufficiently large $R$,~\eqref{phase:finite-time-tail} gives
$u^\beta(t,R)<\eta$ for $0<t\le T_\delta$. Since $m(t)\to\infty$, the equations
\[
 m(t_-(R))+\xi_\beta+z_\eta=R-\delta,
 \qquad
 m(t_+(R))+\xi_\beta+z_\eta=R+\delta
\]
have unique roots in $[T_\delta,\infty)$, with
$T_\delta<t_-(R)<t_+(R)$, once $R$ is sufficiently large.
For $T_\delta\le t\le t_-(R)$, we have
\[
 X_\eta^\beta(t)<m(t)+\xi_\beta+z_\eta+\delta\le R,
\]
hence $u^\beta(t,R)<\eta$. At $t=t_+(R)$, we instead have
$X_\eta^\beta(t)>m(t)+\xi_\beta+z_\eta-\delta=R$, hence
$u^\beta(t,R)>\eta$. Together with the exclusion of arrivals before
$T_\delta$, these inequalities yield
\begin{equation}\label{phase:arrival-bracket}
 t_-(R)\le T_\eta^\beta(R)\le t_+(R).
\end{equation}
Applying the increasing function $m$ to this bracket gives
\[
 R-\xi_\beta-z_\eta-\delta
 \le m(T_\eta^\beta(R))
 \le R-\xi_\beta-z_\eta+\delta.
\]
Since this holds for every $\delta>0$ and all sufficiently large $R$, we obtain
\begin{equation}\label{phase:arrival-inverse}
 m(T_\eta^\beta(R))=R-\xi_\beta-z_\eta+o(1).
\end{equation}

Apply the same argument with $\beta=0$ and subtract the two identities.
The mean value theorem gives a point $\theta_R$ between the two arrival
times such that
\[
 m'(\theta_R)\big(T_\eta^\beta(R)-T_\eta^0(R)\big)
 =\xi_0-\xi_\beta+o(1).
\]
If the arrival times coincide, take $\theta_R$ equal to their common value.
By~\eqref{phase:finite-time-tail}, both arrival times tend to infinity as
$R\to\infty$. Hence $\theta_R\to\infty$ and
$m'(\theta_R)=2-3/(2\theta_R)\to2$, proving~\eqref{phase:delay-formula}.
The level-dependent term $z_\eta$ cancels upon subtraction, so the limiting
delay is independent of $\eta$.
\end{proof}

The phase and delay formulas reduce the remaining parameter-dependence questions to continuity and strict ordering of $\alpha_\beta$, which we address next.

\section{Continuous dependence on the resistance and strict monotonicity}
\label{order:sec}

This section completes the continuity assertion of Theorem~\ref{thm:phase} and proves Theorem~\ref{thm:monotonicity}. Throughout, Assumptions~\ref{hyp:f}--\ref{hyp:data} hold. They suffice for continuity and the endpoint limits below; Assumption~\ref{hyp:concave} is added for the subsequent strict comparison. All solutions start from the same initial datum, which need not be monotone in the transition region. By the phase and delay formulas established in Section~\ref{sec:phase}, it suffices to study the leading-edge amplitude $\alpha_\beta$, since
\begin{equation}
 \tau_\beta=\frac12\log\frac{\alpha_0}{\alpha_\beta}.
 \label{order:tau-alpha}
\end{equation}

\subsection{Continuity and endpoint limits}

Recall the reflected variables $p_\beta,q_\beta,S_\beta,D_\beta$ from
\eqref{eq:fold}. By~\eqref{eq:fold-bc}, $S_\beta$ satisfies a Neumann condition,
whereas $D_\beta$ satisfies a Robin condition with coefficient
$k_\beta=2/\beta$ for $\beta>0$, and a Dirichlet condition for $\beta=0$.
For use here and in the comparison argument below, write the reaction terms in
\eqref{eq:fold-pde} as
\begin{equation}
 F(S,D)=\frac{f(S+D)+f(S-D)}2,\qquad
 H(S,D)=\frac{f(S+D)-f(S-D)}2.
 \label{order:FH}
\end{equation}
All reflected solutions lie in the convex set
$\{(S,D):0\le S-D\le S+D\le1\}$ by Lemma~\ref{lem:reflection-comparison}.
On this set, $F,H$ are bounded and Lipschitz with constants depending only on $f$.

Recall the Gaussian kernel $G_t$ from Section~\ref{sec:preliminaries}.
For $k>0$ and $\beta=2/k$, the antisymmetric part of the transmission kernel
\eqref{eq:barrier-kernel} gives the Robin kernel:
\[
 K_k(t,x,y)=P_\beta(t,x,y)-P_\beta(t,x,-y),\qquad x,y>0.
\]
The corresponding sum is $G_t(x-y)+G_t(x+y)$, the Neumann kernel governing
$S_\beta$. Substituting $r=ks$ in the integral representation of $P_\beta$
therefore gives the heat kernel for $v_x(0)=kv(0)$ as
\begin{equation}
 K_k(t,x,y)=G_t(x-y)+G_t(x+y)
 -2k\int_0^\infty e^{-ks}G_t(x+y+s)\,ds,
 \qquad k\ge0.
 \label{order:robin-kernel}
\end{equation}
Write $K_D(t,x,y)=G_t(x-y)-G_t(x+y)$ for the Dirichlet kernel.
For $k_n\to k>0$, integration in $y$ gives
\[
 \sup_{x\ge0}\int_0^\infty|K_{k_n}(t,x,y)-K_k(t,x,y)|\,dy
 \le 2\int_0^\infty|k_ne^{-k_ns}-ke^{-ks}|\,ds\longrightarrow0,
\]
uniformly in $t>0$. For the Dirichlet limit, use
\[
 K_k-K_D
 =2\int_0^\infty ke^{-ks}
 [G_t(x+y)-G_t(x+y+s)]\,ds.
\]
The Gaussian translation bound
\[
 \int_0^\infty|G_t(x+y)-G_t(x+y+s)|\,dy
 \le C\min\{1,s/\sqrt t\}
\]
holds uniformly for $x\ge0$. Consequently,
\begin{equation}
 \sup_{x\ge0}\int_0^\infty|K_k(t,x,y)-K_D(t,x,y)|\,dy
 \le C\min\{1,(k\sqrt t)^{-1}\},\qquad k>0.
 \label{order:dirichlet-limit}
\end{equation}
These integral estimates apply to bounded initial data without any spatial integrability assumption.

To pass to the nonlinear solutions, let $\mathsf N(t)$ denote the Neumann
heat semigroup, and let $\mathsf R_\beta(t)$ denote the Robin semigroup with
coefficient $2/\beta$ for $\beta>0$, or the Dirichlet semigroup for $\beta=0$.
These semigroups are contractions on $L^\infty(0,\infty)$. Given
$\beta_n\to\beta<\infty$, set
\[
 \varepsilon_n(t)=\|\mathsf R_{\beta_n}(t)-\mathsf R_\beta(t)\|_{L^\infty\to L^\infty}.
\]
The kernel estimates show that $\varepsilon_n\le2$ and
$\varepsilon_n\to0$ uniformly on every compact time interval contained in $(0,\infty)$.
The common initial profiles are
$S_{\mathrm{in}}(x)=D_{\mathrm{in}}(x)=u_0(-x)/2$.
The Duhamel formulas for~\eqref{eq:fold-pde} read
\[
 \begin{split}
 S_\beta(t)&=\mathsf N(t)S_{\mathrm{in}}
 +\int_0^t\mathsf N(t-s)F(S_\beta(s),D_\beta(s))\,ds,\\
 D_\beta(t)&=\mathsf R_\beta(t)D_{\mathrm{in}}
 +\int_0^t\mathsf R_\beta(t-s)H(S_\beta(s),D_\beta(s))\,ds.
 \end{split}
\]
For a fixed finite $T>0$, subtracting these formulas and using the boundedness
and Lipschitz continuity of $F,H$ gives, for $0<t\le T$,
\[
 \begin{split}
 E_n(t)&:=\|S_{\beta_n}(t)-S_\beta(t)\|_\infty
       +\|D_{\beta_n}(t)-D_\beta(t)\|_\infty\\
 &\le A_n(t)+L\int_0^t E_n(s)\,ds,\qquad
 A_n(t)=C\varepsilon_n(t)+C\int_0^t\varepsilon_n(s)\,ds.
 \end{split}
\]
Here $C,L$ depend only on the common data bounds and $f$, and are independent
of $n$ and the resistance values. Gronwall's inequality yields
$E_n(t)\le A_n(t)+L\int_0^t e^{L(t-s)}A_n(s)\,ds$.
Dominated convergence gives $\int_0^T A_n(s)\,ds\to0$, while
$A_n\to0$ uniformly on $[\delta,T]$ for every $0<\delta<T$.
Since $q_\beta=S_\beta-D_\beta$, we conclude that
\begin{equation}
 \sup_{\delta\le t\le T}\sup_{x\ge0}
 |q_{\beta_n}(t,x)-q_\beta(t,x)|\longrightarrow0,
 \qquad \beta_n\to\beta<\infty.
 \label{order:finite-time-cont}
\end{equation}
Boundary values here are right-hand traces. In particular, this proves finite-time continuous dependence at zero resistance as well as at every positive finite resistance.

We now transfer this continuity to the amplitude. Fix $0<\rho<1/3$ and
use the moment $\mathcal J_\beta(T)$ from~\eqref{front:uniform-moment}, with
$\ell(T)=m(T)+T^\rho$. Choose $T$ above the common threshold in
Lemma~\ref{lem:leading-edge}, so that also $\ell(T)>0$.
The Gaussian bound~\eqref{eq:uniform-gaussian} supplies an integrable majorant
for $(x-\ell(T))e^xq_\beta(T,x)$ on $[\ell(T),\infty)$, independent of $\beta$.
Thus~\eqref{order:finite-time-cont} and dominated convergence imply
$\mathcal J_{\beta_n}(T)\to\mathcal J_\beta(T)$ for each such fixed $T$.
The uniform moment estimate gives
\begin{equation}
 2\sqrt\pi\,|\alpha_{\beta_n}-\alpha_\beta|
 \le |\mathcal J_{\beta_n}(T)-\mathcal J_\beta(T)|
       +2C_\rho T^{\rho-1/2}.
 \label{order:amplitude-cont-bound}
\end{equation}
First let $n\to\infty$ and then $T\to\infty$. This proves continuity of
$\alpha_\beta$ at every finite parameter value. Since $\alpha_\beta>0$,
\eqref{order:tau-alpha} proves continuity of $\tau_\beta$ on $[0,\infty)$;
in particular,
\begin{equation}
 \lim_{\beta\downarrow0}\tau_\beta=0.
 \label{order:zero-endpoint}
\end{equation}

The same uniform moment estimate also determines the strong-barrier endpoint,
without using concavity or parameter ordering. For $\beta>0$, the flux defined
in~\eqref{eq:reflection} satisfies
\begin{equation}
 0\le j_\beta(t)
 =\frac{p_\beta(t,0)-q_\beta(t,0)}\beta\le\frac1\beta,
 \qquad t>0.
 \label{order:flux-bound}
\end{equation}
Since $q_\beta(0,x)=0$ for $x>0$ and $f(q_\beta)\le q_\beta$, comparison
with the linear equation $v_t=v_{xx}+v$ with zero initial data and constant
boundary input $-v_x(t,0)=1/\beta$ yields
\begin{align}
 q_\beta(T,x)
 &\le\frac1\beta\int_0^T
 \frac{e^{T-s}}{\sqrt{\pi(T-s)}}
 \exp\left(-\frac{x^2}{4(T-s)}\right)\,ds \notag\\
 &\le\frac{2e^T\sqrt{T/\pi}}\beta
 \exp\left(-\frac{x^2}{4T}\right).
 \label{order:neumann-upper}
\end{align}
Integrating this bound against the weight defining $\mathcal J_\beta(T)$
shows that $\mathcal J_\beta(T)\to0$ as $\beta\to\infty$ for each fixed,
sufficiently large $T$. Since the error in~\eqref{front:uniform-moment}
is uniform over all finite $\beta$, we obtain
\[
 0\le\limsup_{\beta\to\infty}\alpha_\beta
 \le\frac{C_\rho}{2\sqrt\pi}T^{\rho-1/2}.
\]
Now let $T\to\infty$ and use~\eqref{order:tau-alpha} to conclude that
\begin{equation}
 \alpha_\beta\longrightarrow0,\qquad
 \tau_\beta\longrightarrow+\infty
 \qquad(\beta\to\infty).
 \label{order:infinite-endpoint}
\end{equation}

We next establish strict ordering, starting with a pointwise comparison of the solutions.

\subsection{Pointwise comparison under concavity}

We now impose concavity in addition to the assumptions used for continuity.
The reflected system makes its role explicit: it produces nonnegative coupling
coefficients, while the difference in resistance supplies a strictly positive boundary input.

\begin{proposition}\label{prop:pointwise-order}
Under Assumptions~\ref{hyp:f}--\ref{hyp:concave}, solutions with the same initial datum satisfy
\begin{equation}
 q_{\beta_1}(t,x)>q_{\beta_2}(t,x)
 \qquad(t>0,\ x>0),\qquad 0\le\beta_1<\beta_2<\infty.
 \label{order:pointwise}
\end{equation}
\end{proposition}

\begin{proof}
First suppose $0<\beta_1<\beta_2$ and put $k_i=2/\beta_i$, so $k_1>k_2$.
Use the subscript $i$ for the reflected quantities in~\eqref{eq:fold}.
Lemma~\ref{lem:reflection-comparison} gives
\begin{equation}
 D_i(t,x)>0\qquad(t>0,\ x\ge0).
 \label{order:D-positive}
\end{equation}
Set
\begin{equation}
 r=S_1-S_2,\qquad h=D_2-D_1.
 \label{order:rh}
\end{equation}
Both states and the line segment joining them lie in the convex set
$\{(S,D):0\le S-D\le S+D\le1\}$. On this set concavity gives
\begin{equation}
 F_S=H_D=\frac{f'(S+D)+f'(S-D)}2,\qquad
 F_D=H_S=\frac{f'(S+D)-f'(S-D)}2\le0.
 \label{order:jacobian}
\end{equation}
At each $(t,x)$, write
$(S_\theta,D_\theta)=(1-\theta)(S_2,D_2)+\theta(S_1,D_1)$ and define
\[
 A=\int_0^1 F_S(S_\theta,D_\theta)\,d\theta,
 \qquad B=-\int_0^1 F_D(S_\theta,D_\theta)\,d\theta.
\]
These coefficients are bounded and $B\ge0$. The identities in
\eqref{order:jacobian} give the cooperative system
\begin{equation}
 \begin{cases}
 r_t=r_{xx}+Ar+Bh,\\
 h_t=h_{xx}+Br+Ah.
 \end{cases}
 \label{order:cooperative}
\end{equation}
The common initial datum gives zero initial data for $(r,h)$, and the boundary conditions are
\begin{equation}
 r_x(t,0)=0,\qquad
 -h_x(t,0)+k_2h(t,0)=(k_1-k_2)D_1(t,0)>0.
 \label{order:rh-bc}
\end{equation}
For the cooperative comparison, choose $L\ge\|A\|_\infty$ and use the
damped scalar Neumann and Robin semigroups. Both preserve positivity;
the nonnegative Robin boundary input also has a nonnegative response by
scalar comparison. In the Duhamel system the remaining coefficients are
$A+L\ge0$ and $B\ge0$, so successive approximation preserves
nonnegativity and gives $r,h\ge0$. This mild argument applies to the
measurable initial datum without initial compatibility conditions.
For positive times, $h_t-h_{xx}-Ah=Br\ge0$.
The strictly positive boundary input in~\eqref{order:rh-bc} excludes an
identically zero solution on any positive time interval. The strong maximum
principle therefore gives $h(t,x)>0$ for $t>0,x>0$.
Since $q_{\beta_1}-q_{\beta_2}=r+h$, this proves~\eqref{order:pointwise}
when $\beta_1>0$. Strict concavity is not needed: strictness comes from the
boundary input, even if the coupling coefficient $B$ vanishes.

If $\beta_1=0$, use $r=S_0-S_2$ and $h=D_2-D_0$.
The same coefficient construction gives~\eqref{order:cooperative}.
Here $D_0(t,0)=0$ and $D_0>0$ in the interior by
Lemma~\ref{lem:reflection-comparison}; the Hopf lemma gives
$(D_0)_x(t,0)>0$. Thus the boundary condition for $h$ becomes
\begin{equation}
 -h_x(t,0)+k_2h(t,0)=(D_0)_x(t,0)>0.
 \label{order:zero-rh-bc}
\end{equation}
The same mild comparison and positive-time strong maximum principle give
$r\ge0$ and $h>0$, completing the proof.
\end{proof}

\subsection{Strict ordering of the phases}

Proposition~\ref{prop:pointwise-order} and the amplitude limits in
Lemma~\ref{lem:leading-edge} give only
$\alpha_{\beta_1}\ge\alpha_{\beta_2}$.
To obtain strict ordering, we establish a positive lower bound for the
normalized difference in the leading edge and then use the existence of its
limit. The following auxiliary lemma does not require concavity and will
also be used in the weak-barrier analysis.

\begin{lemma}
\label{lem:positive-tail}
Let $W(t,x)>0$ be a classical solution for $t>0,x>0$, bounded on
$[t_0,T_1]\times(0,\infty)$ for every $0<t_0<T_1<\infty$, and satisfying
\begin{equation}
 W_t=W_{xx}+c(t,x)W,
 \label{order:W-pde}
\end{equation}
where $c$ is bounded. Let $m$ be the reference position defined in~\eqref{eq:m}.
Assume that, for all sufficiently large $t$ and all $z\geq0$,
\begin{equation}
 c(t,m(t)+z)\geq1-Ct e^{-z}.
 \label{order:potential-tail}
\end{equation}
Then, for every $0<\gamma<1/3$,
\begin{equation}
 \liminf_{t\to\infty}
 \frac{e^{t^\gamma}W(t,m(t)+t^\gamma)}{t^\gamma}>0.
 \label{order:positive-amplitude}
\end{equation}
The positive lower bound may depend on $W$, $\gamma$, and the coefficient
bounds and starting time in~\eqref{order:potential-tail}; no uniform lower
bound over a family of positive solutions is asserted. No initial trace at
$t=0$ or boundary condition at $x=0$ is required.
\end{lemma}

\begin{proof}
Fix $0<\gamma<1/3$, choose $0<\kappa<\gamma$, and set
\begin{equation}
 B(t)=t^\kappa,\qquad
 P(t,z)=t^{-3/2}e^zW(t,m(t)+z),\qquad
 \widetilde P(t,y)=P(t,B(t)+y).
 \label{order:W-transform}
\end{equation}
A direct calculation gives
\begin{equation}
 \widetilde P_t=\widetilde P_{yy}+b(t)\widetilde P_y
 +(c(t,m(t)+B(t)+y)-1)\widetilde P,
 \qquad b(t)=B'(t)-\frac3{2t}.
 \label{order:W-transformed-pde}
\end{equation}
For all sufficiently large $t$ and $y>0$, the potential is bounded below by
$-\omega(t)$, where
\begin{equation}
 \omega(t)=Ct e^{-t^\kappa}.
 \label{order:omega}
\end{equation}
This function is integrable at infinity.

Choose $T>1$ sufficiently large that $m(t)+B(t)>0$, the potential bound
\eqref{order:potential-tail} holds, $b(t)\ge0$, and
$0<t^{\kappa-1/2}\le1/2$ for every $t\ge T$. Fix this $T$ before choosing
the initial seed and the normalization below.
We construct a subsolution from the Dirichlet heat profile
$yr^{-3/2}e^{-y^2/(4r)}$. The exponential factor $e^{-b(t)y/2}$ removes
the first spatial derivative from the residual. Slowing the heat time
slightly below unit speed creates a negative quadratic term that absorbs
the term involving $b'(t)y$. Accordingly, define
\begin{equation}
 \varepsilon(t)=t^{\kappa-1/2},\qquad
 r(T)=1,\qquad r'(t)=1-\varepsilon(t).
 \label{order:time-change}
\end{equation}
Then $r>0$ and $r(t)/t\to1$. Define
\begin{equation}
 Q(t,y)=a(t)e^{-b(t)y/2}
 y\,r(t)^{-3/2}\exp\left(-\frac{y^2}{4r(t)}\right),
 \label{order:subsolution}
\end{equation}
The normalization $a(T)>0$ will be chosen from an initial Gaussian lower bound for $W$. For now, prescribe
\begin{equation}
 \frac{a'(t)}{a(t)}=-\left[
 \omega(t)+\frac{b(t)^2}{4}
 +\frac{3\varepsilon(t)}{2r(t)}
 +\frac{r(t)^2b'(t)^2}{4\varepsilon(t)}\right].
 \label{order:a-choice}
\end{equation}
Substitution gives the exact identity
\begin{equation}
 \frac{Q_t-Q_{yy}-bQ_y+\omega Q}{Q}
 =-\left(\frac{\sqrt\varepsilon\,y}{2r}
 +\frac{rb'}{2\sqrt\varepsilon}\right)^2\leq0.
 \label{order:subsolution-residual}
\end{equation}
Since
\begin{equation}
 b(t)^2=O(t^{2\kappa-2}),\qquad
 \frac{\varepsilon(t)}{r(t)}
 +\frac{r(t)^2b'(t)^2}{\varepsilon(t)}
 =O(t^{\kappa-3/2}),
 \label{order:integrable-losses}
\end{equation}
and $2\kappa-2<-1$, $\kappa-3/2<-1$, all terms in~\eqref{order:a-choice}
are integrable on $[T,\infty)$. Thus, for any $a(T)>0$, the limit
$a_\infty$ equals $a(T)$ times the exponential of minus the finite integral
of the nonnegative loss terms, and is strictly positive.

To establish the initial comparison, fix $x_0=m(T)+B(T)>0$.
Positivity and continuity of $W(T-1,\cdot)$ allow us to choose a nonzero,
nonnegative smooth seed supported in $(x_0+1,x_0+2)$ and lying below
$W(T-1,\cdot)$.
Evolve this seed on the fixed half-line $x>x_0$ under the heat equation
with potential $-\|c\|_\infty$ and zero boundary data.
The comparison principle and the explicit Dirichlet kernel give
\begin{equation}
 W(T,x_0+y)\geq C_Ty e^{-y^2/4}\qquad(y>0).
 \label{order:initial-lower}
\end{equation}
Indeed, for $1\leq s\leq2$,
\begin{equation}
 G_1(y-s)-G_1(y+s)
 =\frac{2}{\sqrt{4\pi}}e^{-(y^2+s^2)/4}\sinh\frac{ys}{2}
 \geq c\,y e^{-y^2/4}.
 \label{order:dirichlet-seed-kernel}
\end{equation}
By~\eqref{order:initial-lower},
$\widetilde P(T,y)\geq C'_Ty e^{-y^2/4+y}$.
Since $r(T)=1$ and $b(T)\ge0$, choosing $0<a(T)\le C'_T$ ensures
$Q(T,y)\le\widetilde P(T,y)$ for all $y>0$. This fixes $a$ and its positive
limit $a_\infty$.

To justify comparison on the unbounded moving domain, return to the original variables and set
\[
 \underline W(t,x)=t^{3/2}e^{-(x-m(t))}
 Q\bigl(t,x-m(t)-B(t)\bigr),\qquad x>m(t)+B(t).
\]
By~\eqref{order:potential-tail} and~\eqref{order:subsolution-residual},
$\underline W_t-\underline W_{xx}-c\underline W\le0$.
The initial inequality holds by the choice of $a(T)$, and
$\underline W(t,m(t)+B(t))=0<W(t,m(t)+B(t))$.
The moving boundary is smooth and stays in $x>0$. Moreover, $b\ge0$ and
the Gaussian factor in $Q$ ensure that $\underline W$ is bounded on every
strip $T\le t\le T_1<\infty$, as is $W$ by assumption.
The comparison principle for bounded solutions on this moving half-line
therefore gives $\underline W\le W$, equivalently $Q\le\widetilde P$.
No boundary condition at spatial infinity is required in this bounded class.

Finally, set $y=t^\gamma-t^\kappa$. Since $\kappa<\gamma<1/3$,
\begin{equation}
 \frac{y}{t^\gamma}\to1,\qquad b(t)y\to0,\qquad
 \frac{y^2}{r(t)}\to0,\qquad \frac{r(t)}t\to1.
 \label{order:matching-limits}
\end{equation}
Equations~\eqref{order:subsolution} and~\eqref{order:W-transform} yield
\[
 \liminf_{t\to\infty}
 \frac{e^{t^\gamma}W(t,m(t)+t^\gamma)}{t^\gamma}
 \geq a_\infty>0.
\]
\end{proof}

\begin{proof}[Proof of Theorem~\ref{thm:monotonicity}]
Fix $0\le\beta_1<\beta_2<\infty$ and set $W=q_{\beta_1}-q_{\beta_2}$.
Proposition~\ref{prop:pointwise-order} gives $0<W\le1$ on $t>0,x>0$.
Positive-time interior regularity makes $W$ a classical solution of
\begin{equation}
 W_t=W_{xx}+c(t,x)W,
 \qquad
 c(t,x)=\int_0^1f'\bigl(q_{\beta_2}(t,x)+\theta W(t,x)\bigr)\,d\theta.
 \label{order:difference-equation}
\end{equation}
The arguments of $f'$ lie in $[0,1]$, so $|c|\le\|f'\|_{L^\infty([0,1])}$.
Moreover, $f'(0)=1$ and $f\in C^2([0,1])$ imply
\[
 |f'(s)-1|\le\|f''\|_{L^\infty([0,1])}s,\qquad 0\le s\le1.
\]
Together with~\eqref{eq:uniform-leading-bound}, this gives, for all sufficiently
large $t$ and all $z\ge0$,
\begin{equation}
 q_{\beta_i}(t,m(t)+z)\leq Ct e^{-z},\qquad
 c(t,m(t)+z)\geq1-Ct e^{-z}.
 \label{order:difference-potential}
\end{equation}
All hypotheses of Lemma~\ref{lem:positive-tail} are therefore satisfied.
Here concavity enters only through the pointwise comparison giving $W>0$.
Fix any $0<\gamma<1/3$. Combining the positive lower bound with the amplitude
limits in Lemma~\ref{lem:leading-edge}, we obtain
\begin{equation}
 \alpha_{\beta_1}-\alpha_{\beta_2}
 =\lim_{t\to\infty}
 \frac{e^{t^\gamma}
 \bigl(q_{\beta_1}-q_{\beta_2}\bigr)(t,m(t)+t^\gamma)}{t^\gamma}
 >0.
 \label{order:strict-alpha}
\end{equation}
The phase formula $\xi_\beta=\log(\alpha_\beta/A_*)$ now gives
\begin{equation}\label{order:strict-phase}
 \xi_{\beta_1}-\xi_{\beta_2}
 =\log\frac{\alpha_{\beta_1}}{\alpha_{\beta_2}}>0.
\end{equation}
Thus the phase is strictly decreasing with resistance, and
\eqref{order:tau-alpha} yields
\begin{equation}
 \tau_{\beta_2}-\tau_{\beta_1}
 =\frac12\log\frac{\alpha_{\beta_1}}{\alpha_{\beta_2}}>0
 \qquad(0\leq\beta_1<\beta_2<\infty).
 \label{order:strict-tau}
\end{equation}
Taking $\beta_1=0$ and using $\tau_0=0$ gives $\tau_\beta>0$ for every $\beta>0$.
\end{proof}

These results determine the strict direction of the barrier's effect on the
phase and delay. The next section quantifies its first-order behavior near
zero resistance.

\section{The weak-barrier limit and the first variation}\label{sec:small-beta}

This section proves Theorem~\ref{thm:small-beta} and characterizes its coefficient through a linear evolution problem. We write $U=u^0$, while $U_*$ continues to denote the normalized minimal-speed traveling wave. Assumptions~\ref{hyp:f}--\ref{hyp:data} suffice for the finite-time first variation. We add Assumption~\ref{hyp:concave} in the subsequent analysis to obtain the uniform estimates needed to pass to the far-field amplitude.

\subsection{The first variation on finite time intervals}

\begin{proposition}\label{small:finite-variation}
Under Assumptions~\ref{hyp:f}--\ref{hyp:data}, there exists a unique first variation $w$, bounded on $[0,T]\times(\mathbb R\setminus\{0\})$ for every $T>0$, such that
\begin{equation}\label{small:finite-expansion}
 \sup_{0\le t\le T}\|u^\beta(t,\cdot)-U(t,\cdot)-\beta w(t,\cdot)\|_{L^\infty(\mathbb R\setminus\{0\})}=o(\beta)
 \qquad(\beta\downarrow0).
\end{equation}
The function $w$ is classical in the interior of each half-line for positive times, has one-sided traces of $w$ and $w_x$ at the interface, and satisfies
\begin{equation}\label{small:variation-problem}
 \begin{cases}
 w_t=w_{xx}+f'(U)w,&x\ne0,\ t>0,\\
 [w](t)=U_x(t,0),&t>0,\\
 [w_x](t)=0,&t>0,\\
 w(0,x)=0,&x\ne0.
 \end{cases}
\end{equation}
The initial condition is attained in $L^\infty(\mathbb R\setminus\{0\})$ as $t\downarrow0$. This problem uniquely characterizes $w$ in the stated bounded solution class.
\end{proposition}

\begin{proof}
Use the transmission kernel~\eqref{eq:barrier-kernel}, with $P_0(t)=G_t$.
We use the same notation for these kernels and their integral operators.
Differentiating its integral representation from the right at $\beta=0$ gives
\begin{equation}\label{small:kernel-variation}
 \mathcal B_t(x,y)=-\frac{\operatorname{sgn}(x)\operatorname{sgn}(y)}2G'_t(|x|+|y|).
\end{equation}
At the interface, signs are interpreted through one-sided limits. Write
$\mathcal D_\beta(t)=(P_\beta(t)-G_t)/\beta$ for $\beta>0$.
The mean-value formula and integration of $|G'_t|$ give
\begin{equation}\label{small:kernel-quotient}
 \|\mathcal D_\beta(t)\|_{L^\infty\to L^\infty},\quad
 \|\mathcal B_t\|_{L^\infty\to L^\infty}
 \le\frac1{\sqrt{4\pi t}}.
\end{equation}
Subtracting $G'_t(|x|+|y|)$ from its translated value in the mean-value formula also gives
\[
 \|\mathcal D_\beta(t)-\mathcal B_t\|_{L^\infty\to L^\infty}
 \le C\beta\|G''_t\|_{L^1(\mathbb R)}\le C\beta/t.
\]
In particular, convergence holds in operator norm uniformly on $[\delta,T]$ for every $0<\delta<T$.

The initial term admits a stronger bound near $t=0$. With $d=-b>0$ from
Assumption~\ref{hyp:data}, $u_0$ is supported in $(-\infty,-d]$.
Integrating $|G'_t|$ over this region yields
\begin{equation}\label{small:initial-kernel-bound}
 \|\mathcal D_\beta(t)u_0\|_\infty+\|\mathcal B_tu_0\|_\infty
 \le Ct^{-1/2}e^{-d^2/(4t)},\qquad \beta>0.
\end{equation}
Splitting $[0,T]$ at a small positive time and then using operator-norm
convergence away from zero shows that
\[
 \sup_{0<t\le T}\|(\mathcal D_\beta(t)-\mathcal B_t)u_0\|_\infty\longrightarrow0.
\]
For the reaction term, split the time integral according to whether
$t-s<\delta$ or $t-s\ge\delta$. The first part is bounded by $C\sqrt\delta$
uniformly for $t\le T$, by~\eqref{small:kernel-quotient}; the second tends to
zero by operator-norm convergence. Hence
\begin{equation}\label{small:forcing-convergence}
 \sup_{0\le t\le T}\left\|
 \int_0^t(\mathcal D_\beta(t-s)-\mathcal B_{t-s})f(U(s))\,ds
 \right\|_\infty\longrightarrow0.
\end{equation}

Fix $\beta_*>0$ and put $W_\beta=(u^\beta-U)/\beta$, with $W_\beta(0)=0$.
Subtracting the two nonlinear Duhamel formulas gives
\[
 W_\beta(t)=\mathcal D_\beta(t)u_0
 +\int_0^t\mathcal D_\beta(t-s)f(U(s))\,ds
 +\int_0^t P_\beta(t-s)\frac{f(u^\beta(s))-f(U(s))}{\beta}\,ds.
\]
The contraction bound for $P_\beta$, the Lipschitz continuity of $f$, and
\eqref{small:initial-kernel-bound} imply, by Gronwall's inequality,
\begin{equation}\label{small:finite-quotient-bound}
 \sup_{0<\beta\le\beta_*}\sup_{0\le t\le T}
 \|W_\beta(t)\|_\infty\le C_T.
\end{equation}
Here and below the constants may depend on $f,u_0,T$, but not on
$\beta\in(0,\beta_*]$.
Define $w$ as the unique bounded solution, on each finite time interval, of
\begin{equation}\label{small:variation-duhamel}
 \begin{split}
 w(t)={}&\mathcal B_tu_0+\int_0^t\mathcal B_{t-s}f(U(s))\,ds\\
 &+\int_0^tG_{t-s}\bigl[f'(U(s))w(s)\bigr]\,ds.
 \end{split}
\end{equation}
Existence and uniqueness follow by successive approximation for this linear
Volterra equation; uniqueness makes these solutions consistent as $T$ varies.
The first two terms tend to zero in $L^\infty$ as $t\downarrow0$, by
\eqref{small:initial-kernel-bound} and~\eqref{small:kernel-quotient}, so
$\|w(t)\|_\infty\to0$ as well.

Taylor's formula gives
\[
 \left\|\frac{f(u^\beta)-f(U)}\beta-f'(U)W_\beta\right\|_\infty
 \le\frac12\|f''\|_\infty\beta\|W_\beta\|_\infty^2.
\]
When the integral equations for $W_\beta$ and $w$ are subtracted, the
initial and reaction forcing errors tend to zero uniformly on $[0,T]$
by the estimates above. The additional semigroup error is bounded by
\[
 \left\|\int_0^t(P_\beta(t-s)-G_{t-s})[f'(U(s))w(s)]\,ds\right\|_\infty
 \le C_T\beta\sqrt T.
\]
The Taylor remainder contributes at most $C_T\beta T$, and the remaining
term is bounded by $\|f'\|_\infty\int_0^t\|W_\beta(s)-w(s)\|_\infty\,ds$.
Another application of Gronwall's inequality proves uniform convergence
$W_\beta\to w$ on $[0,T]$, establishing~\eqref{small:finite-expansion}.

We next record a spatial bound needed for the later amplitude analysis.
The same kernel representation gives
\[
 \int_{\mathbb R}|\mathcal D_\beta(t,x,y)|\,dy\le G_t(x),
 \qquad 0\le P_\beta(t,x,y)\le2G_t(x-y).
\]
With $M=\|f'\|_\infty$, the difference-quotient equation therefore yields
\[
 |W_\beta(t,x)|\le CG_t(x)+C\int_0^tG_r(x)\,dr
 +2M\int_0^t(G_{t-s}*|W_\beta(s)|)(x)\,ds.
\]
Iterating this inequality and using the Gaussian convolution identity gives
\[
 |W_\beta(T,x)|\le Ce^{2MT}G_T(x)
 +C\int_0^T e^{2Mr}G_r(x)\,dr.
\]
Since $e^{-x^2/(4r)}\le e^{-x^2/(4T)}$ for $r\le T$ and
$\int_0^T r^{-1/2}\,dr=2\sqrt T$, we conclude that
\begin{equation}\label{small:finite-gaussian}
 |W_\beta(T,x)|\le C_Te^{-x^2/(4T)},
 \qquad 0<\beta\le\beta_*,\quad x\ne0.
\end{equation}
The same bound holds for $w(T,x)$ by passage to the limit.
For any fixed, sufficiently large starting time $T_0$ and fixed cutoff
exponent $\rho$, the negative-cutoff normalization $Z_{W_\beta}$ in
Corollary~\ref{cor:signed-amplitudes} consequently satisfies
$\sup_{0<\beta\le\beta_*}\|Z_{W_\beta}(\log T_0,\cdot)\|_{1/5}<\infty$.
Indeed, after that change of variables the bound is a Gaussian with quadratic
exponent $-y^2/4$ and a linear term, which is absorbed by the weaker weight
$e^{y^2/5}$. This is an initial weighted bound at the fixed time $T_0$.

It remains to identify the equation and interface traces of $w$.
For a bounded Lipschitz function $\phi$, direct integration of
\eqref{small:kernel-variation} gives, for $r>0$,
\begin{equation}\label{small:kernel-traces}
 [\mathcal B_r\phi]=\partial_x(G_r*\phi)(0),\qquad
 [\partial_x\mathcal B_r\phi]=0.
\end{equation}
These trace identities also hold for bounded measurable $\phi$ when $r>0$
is fixed. To justify taking traces under the time integral in
\eqref{small:variation-duhamel}, fix $t>0$ and split it at $s=t/2$.
For $s\le t/2$, all kernel derivatives involved are integrable with bounds
away from $r=t-s=0$. For $s\ge t/2$, $f(U(s,\cdot))$ has a common spatial
Lipschitz bound. In fact, writing $a=|x|$ gives on either side of the interface
\[
 \partial_x\mathcal B_r\phi(x)
 =-\frac12\int_0^\infty G''_r(a+y)[\phi(y)-\phi(-y)]\,dy,
\]
and thus
\[
 \sup_{x\ne0}|\partial_x\mathcal B_r\phi(x)|
 \le C\operatorname{Lip}(\phi)r^{-1/2}.
\]
This integrable bound permits passage to both derivative traces; the traces
of $\mathcal B_r\phi$ itself are controlled by~\eqref{small:kernel-quotient}.
The last, ordinary heat-potential term in~\eqref{small:variation-duhamel}
has continuous value and first spatial derivative across zero, since its
source is bounded and $\|G'_r\|_1\le Cr^{-1/2}$.
Taking jumps in~\eqref{small:variation-duhamel} and using the Duhamel formula
for $U$ therefore gives
\[
 [w](t)=\partial_x(G_t*u_0)(0)
 +\int_0^t\partial_x(G_{t-s}*f(U(s)))(0)\,ds=U_x(t,0),
 \qquad [w_x](t)=0.
\]
Away from the interface, $\mathcal B_r$ solves the heat equation and its
initial action vanishes locally as $r\downarrow0$. The integral equation
and interior parabolic regularity then give $w_t=w_{xx}+f'(U)w$ on each
half-line for positive times. Equivalently, in distributions,
\begin{equation}\label{small:distributional-variation}
 w_t=w_{xx}+f'(U)w-U_x(t,0)\delta'_0.
\end{equation}
Indeed, the distributional second derivative contributes
$[w_x]\delta_0+[w]\delta'_0$ in addition to the piecewise second derivative.
The difference of two solutions in the stated class has zero jumps, solves
the homogeneous full-line equation with bounded potential $f'(U)$, and has
zero initial data. Full-line uniqueness therefore proves the asserted
uniqueness for~\eqref{small:variation-problem}.
\end{proof}

Let $\mathcal K(t,x;s,y)$ be the full-line fundamental solution of
$\partial_t-\partial_{xx}-f'(U(t,x))$, for $0<s<t$.
The distributional source in~\eqref{small:distributional-variation} gives
\begin{equation}\label{small:variation-green}
 w(t,x)=\int_0^tU_x(s,0)\partial_y\mathcal K(t,x;s,0)\,ds,\qquad x\ne0.
\end{equation}
The sign follows from the action of $-U_x(s,0)\delta'_0$ on the source
variable $y$. The interface traces are obtained by taking $x\to0^\pm$
after time integration; the singularity at the upper endpoint produces
the jump in~\eqref{small:variation-problem}.

These finite-time estimates allow constants to grow with $T$, so they do
not by themselves justify interchanging $\beta\downarrow0$ and $t\to\infty$.
We next establish uniform decay of the interface flux to supply the
long-time control needed for the amplitude analysis.

\subsection{Flux decay uniformly in the parameter}

For the remainder of Section~\ref{sec:small-beta}, Assumptions~\ref{hyp:f}--\ref{hyp:concave} hold.
Fix any $\beta_*>0$ and consider $0<\beta\le\beta_*$.
Recall the reflected variables from~\eqref{eq:fold} and denote the common
interface derivative by
\[
 d_\beta(t)=u_x^\beta(t,0^\pm)=-(D_\beta)_x(t,0)=-j_\beta(t).
\]

Thus $d_\beta$ is the signed derivative, opposite to the positive incoming
flux $j_\beta$ defined in~\eqref{eq:reflection}.

\begin{lemma}\label{small:uniform-flux}
Under Assumptions~\ref{hyp:f}--\ref{hyp:concave}, $d_\beta\in C^1((0,\infty))$
for each $\beta>0$. For any fixed $\beta_*>0$, there exist $T,C,\omega>0$,
depending on $f,u_0,\beta_*$ but independent of $\beta\in(0,\beta_*]$, such that
\begin{equation}\label{small:flux-decay}
 |d_\beta(t)|+|d'_\beta(t)|\le Ce^{-\omega t},\qquad t\ge T.
\end{equation}
\end{lemma}

\begin{proof}
The parameter order and the reflection comparison give
\[
 q_\beta\ge q_{\beta_*},\qquad p_\beta\ge q_\beta.
\]
Fix $0<c_{\rm in}<c_{\rm out}<2$. Spatial monotonicity gives
\[
 u^\beta(t,x)\ge q_{\beta_*}(t,|x|)
 \ge q_{\beta_*}(t,c_{\rm out}t),\qquad |x|\le c_{\rm out}t,
\]
with one-sided traces at zero. Since $c_{\rm out}<2$, phase convergence for
the fixed parameter $\beta_*$ makes the last expression tend to $1$.
Thus, for every sufficiently small $\delta>0$, there is a common starting
time beyond which
\[
 1-\delta\le u^\beta(t,x)\le1\qquad (|x|\le c_{\rm out}t).
\]
Choose $a_*>0$ so that
\[
 f(s)\ge a_*(1-s),\qquad f'(s)\le-a_*\qquad (1-\delta\le s\le1).
\]
Thus $1-u^\beta$ satisfies $(\partial_t-\partial_{xx}+a_*)(1-u^\beta)\le0$ in this interval. Choose $b_*>0$ such that $b_*^2+b_*c_{\rm out}<a_*$. The even function
\begin{equation}\label{small:saturation-screen}
 \Psi(t,x)=C e^{-a_*(t-T)}+C e^{-b_*c_{\rm out}t}\cosh(b_*x)
\end{equation}
satisfies
\[
 (\partial_t-\partial_{xx}+a_*)\Psi
 =C(a_*-b_*c_{\rm out}-b_*^2)e^{-b_*c_{\rm out}t}\cosh(b_*x)\ge0.
\]
It is continuous at zero with $\Psi_x(t,0)=0$, so it satisfies the interface
conditions for every $\beta$. Choose $T$ above the common starting time.
The first term controls the initial section $t=T$, while
$e^{-b_*c_{\rm out}t}\cosh(b_*c_{\rm out}t)\ge1/2$ at each lateral endpoint.
Thus a common choice of $C$ bounds $1-u^\beta$ on the parabolic boundary.
On $|x|\le c_{\rm in}t$, the spatial term is at most
$Ce^{-b_*(c_{\rm out}-c_{\rm in})t}$. The comparison principle gives
\begin{equation}\label{small:inner-saturation}
 1-u^\beta(t,x)\le Ce^{-\omega t},\qquad |x|\le c_{\rm in}t.
\end{equation}

Let $\mathsf S_\beta(r)$ denote the transmission heat operator, and write
$F_\beta(t)=f(u^\beta(t))$. Differentiating~\eqref{eq:barrier-kernel}
in the heat time, each Gaussian term is controlled after spatial integration
by $\|\partial_rG_r\|_{L^1(\mathbb R)}\le C/r$.
The translated Gaussian terms have the same bound, and their averaging
weight has total mass $\int_0^\infty e^{-s}\,ds=1$.
Consequently, for $r,h>0$,
\begin{equation}\label{small:heat-time-bounds}
 \|\partial_r\mathsf S_\beta(r)\|_{\infty\to\infty}\le \frac Cr,
 \qquad
 \|\mathsf S_\beta(r+h)-\mathsf S_\beta(r)\|_{\infty\to\infty}
 \le C\min\{1,h/r\},
\end{equation}
uniformly in $\beta$. For the second estimate, integrate the first when
$h\le r$, and otherwise use the contraction bound. Apply Duhamel's formula
on $[t-1,t+h]$, keeping the lower endpoint fixed. Boundedness of $u^\beta$
and $F_\beta$ gives, for $t\ge1$ and $0<h\le1$,
\begin{equation}\label{small:uniform-time-modulus}
 \begin{split}
 \|u^\beta(t+h)-u^\beta(t)\|_\infty
 &\le Ch+C\int_0^1\min\{1,h/r\}\,dr\\
 &\le Ch(1+|\log h|)\le C_\theta h^\theta,
 \qquad 0<\theta<1.
 \end{split}
\end{equation}
The Lipschitz continuity of $f$ gives the same temporal modulus for
$F_\beta$. To differentiate at time $t\ge2$, again treat $t-1$ as a fixed
starting time and subtract the terminal reaction value. The resulting
cancellation is
\begin{equation}\label{small:time-cancellation}
 \begin{split}
 u_t^\beta(t)={}&(\partial_r\mathsf S_\beta)(1)u^\beta(t-1)
              +\mathsf S_\beta(1)F_\beta(t)\\
 &+\int_0^1(\partial_r\mathsf S_\beta)(r)
       [F_\beta(t-r)-F_\beta(t)]\,dr.
 \end{split}
\end{equation}
The last integrand has norm at most $Cr^{\theta-1}$. Hence
\begin{equation}\label{small:uniform-time-derivative}
 \sup_{0<\beta\le\beta_*}\sup_{t\ge2}
       \|u_t^\beta(t)\|_\infty\le C.
\end{equation}
These estimates use the invariant bound $0\le u^\beta\le1$ and the boundedness
and Lipschitz continuity of $f$ on $[0,1]$. They require neither spatial
regularity of the initial datum nor concavity.

For each fixed $\beta$, the positive-time differences
$(u^\beta(t+h)-u^\beta(t))/h$ satisfy the same homogeneous interface law,
with reaction coefficient
$\int_0^1f'((1-\theta)u^\beta(t)+\theta u^\beta(t+h))\,d\theta$.
Passing to the limit at positive times yields the linear evolution for
$v_\beta=u_t^\beta$:
\[
 (v_\beta)_t=(v_\beta)_{xx}+f'(u^\beta)v_\beta,\qquad
 [v_\beta]=\beta(v_\beta)_x(t,0^\pm),\qquad [(v_\beta)_x]=0.
\]
The identities hold first in the mild sense on intervals separated from
zero; fixed-parameter parabolic regularity with $f\in C^2$ then gives the
one-sided traces. This regularity is used only to justify the equation and
boundary identities. Since $f'(u^\beta)\le-a_*$ on the outer moving interval,
$\Psi$ from~\eqref{small:saturation-screen} is also a supersolution for
$\partial_t-\partial_{xx}-f'(u^\beta)$ there. Compare both $v_\beta$ and
$-v_\beta$ with $\Psi$. The common global bound
in~\eqref{small:uniform-time-derivative} controls the initial data and
lateral endpoints. Thus, after increasing the common starting time,
\begin{equation}\label{small:inner-time-decay}
 |u_t^\beta(t,x)|\le Ce^{-\omega t},\qquad |x|\le c_{\rm in}t.
\end{equation}
The constants are obtained by comparison and do not depend on a
boundary regularity estimate as $\beta\downarrow0$.

We next control the Robin traces directly. Let $k=2/\beta$, let
$\mathsf T_k(s)$ be the half-line heat operator with condition
$v_x(0)=kv(0)$, and denote its kernel by $K_k$. The
formula~\eqref{order:robin-kernel} and integration by parts give
\begin{equation}\label{small:robin-trace-kernel}
 kK_k(s,0,y)=-2k\int_0^\infty e^{-kr}G'_s(y+r)\,dr\ge0,
 \qquad
 \int_0^\infty kK_k(s,0,y)\,dy\le\frac1{\sqrt{\pi s}}.
\end{equation}
More precisely, for every $R\ge0$,
\begin{equation}\label{small:robin-trace-tail}
 \int_R^\infty kK_k(s,0,y)\,dy
 =2k\int_0^\infty e^{-kr}G_s(R+r)\,dr
 \le\frac{e^{-R^2/(4s)}}{\sqrt{\pi s}}.
\end{equation}
Here $ke^{-kr}\,dr$ is a probability measure, so the bounds are
independent of $k$.
Set
\[
 N_\beta=\frac{f(p_\beta)-f(q_\beta)}2,\qquad
 M_\beta=\frac{f'(p_\beta)(p_\beta)_t-f'(q_\beta)(q_\beta)_t}2.
\]
Time derivatives are written explicitly as $\partial_t D_\beta$; the dot in $\dot\alpha_0$ denotes the right derivative with respect to resistance.
The folded equations and their positive-time difference quotients yield
\[
 \begin{aligned}
 (D_\beta)_t&=(D_\beta)_{xx}+N_\beta,
 & (D_\beta)_x(t,0)&=kD_\beta(t,0),\\
 (\partial_t D_\beta)_t&=(\partial_t D_\beta)_{xx}+M_\beta,
 & (\partial_t D_\beta)_x(t,0)&=k\partial_t D_\beta(t,0).
 \end{aligned}
\]
In particular, on the most recent unit time interval,
\begin{equation}\label{small:differentiated-robin-mild}
 \begin{split}
 D_\beta(t)&=\mathsf T_k(1)D_\beta(t-1)
       +\int_0^1\mathsf T_k(s)N_\beta(t-s)\,ds,\\
 \partial_t D_\beta(t)&=\mathsf T_k(1)\partial_t D_\beta(t-1)
       +\int_0^1\mathsf T_k(s)M_\beta(t-s)\,ds.
 \end{split}
\end{equation}
The definitions give
\[
 |N_\beta|\le\|f'\|_\infty|D_\beta|,\qquad
 |\partial_t D_\beta|+|M_\beta|
 \le C\bigl(|(p_\beta)_t|+|(q_\beta)_t|\bigr).
\]
Thus $D_\beta,N_\beta,\partial_t D_\beta,M_\beta$ are bounded by a common constant
on the half-line for sufficiently large times. Moreover,
$|D_\beta|\le((1-p_\beta)+(1-q_\beta))/2$.
By~\eqref{small:inner-saturation} and~\eqref{small:inner-time-decay}, all four
functions are bounded by $Ce^{-\omega t}$ on
$0\le y\le R_t:=c_{\rm in}(t-1)$ throughout $[t-1,t]$, once this interval
lies beyond the common decay threshold.
Multiply the boundary traces in~\eqref{small:differentiated-robin-mild}
by $k$ and split the integrals at $R_t$. Equations
\eqref{small:robin-trace-kernel}--\eqref{small:robin-trace-tail} give
\[
 \begin{split}
 k|D_\beta(t,0)|+k|\partial_t D_\beta(t,0)|
 &\le Ce^{-\omega t}\left(1+\int_0^1s^{-1/2}\,ds\right)
       +Ce^{-R_t^2/4}\\
 &\quad+C\int_0^1s^{-1/2}e^{-R_t^2/(4s)}\,ds
 \le Ce^{-\omega't}.
 \end{split}
\]
For each fixed $\beta$, the positive-time regularity used above also
permits differentiation of the boundary value $D_\beta(t,0)$ in time,
with derivative $\partial_t D_\beta(t,0)$. The Robin condition therefore gives
\[
 d_\beta(t)=-kD_\beta(t,0),\qquad
 d'_\beta(t)=-k\partial_t D_\beta(t,0).
\]
In particular $d_\beta\in C^1((0,\infty))$. The trace estimates prove
\eqref{small:flux-decay} after reducing the common decay rate if necessary;
no uniform estimate for $u_{tx}^\beta$ is needed.
\end{proof}

For the difference quotient $W_\beta$, the interface law
gives $[W_\beta]=d_\beta$. Removing this jump by a localized spatial lifting
produces source terms involving both $d_\beta$ and $d'_\beta$.
The lemma makes these terms uniformly exponentially small at large times,
which will be used to control the difference quotient in the next subsection.

\subsection{Leading-edge control of the difference quotient and interchange of limits}

Retain $W_\beta=(u^\beta-U)/\beta$ and define the secant coefficient
\begin{equation}\label{small:quotient-coefficient}
 c_\beta(t,x)=\int_0^1f'\bigl(U(t,x)+\theta(u^\beta(t,x)-U(t,x))\bigr)\,d\theta.
\end{equation}
Then $(W_\beta)_t=(W_\beta)_{xx}+c_\beta W_\beta$ away from the interface, with
\begin{equation}\label{small:quotient-jump}
 [W_\beta]=d_\beta,\qquad [(W_\beta)_x]=0.
\end{equation}
Concavity gives $c_\beta\le1$. We verify the common bounds required by
Corollary~\ref{cor:signed-amplitudes} to obtain a moment error that remains
uniform after division by $\beta$. This is the estimate needed to transfer
finite-time differentiability to the leading-edge amplitude.

\begin{lemma}\label{small:scaled-moment}
Under Assumptions~\ref{hyp:f}--\ref{hyp:concave}, fix $\beta_*>0$ and
$0<\rho<1/3$, and use $\mathcal J_\beta$ from~\eqref{front:uniform-moment}
with $\ell(t)=m(t)+t^\rho$.
There are constants $C,T$, depending on $f,u_0,\beta_*,\rho$ but not on
$0<\beta\le\beta_*$, such that, for $t\ge T$,
\begin{equation}\label{small:moment-uniform-error}
 \left|\frac{\mathcal J_\beta(t)-\mathcal J_0(t)}\beta
 -2\sqrt\pi\frac{\alpha_\beta-\alpha_0}\beta\right|
 \le Ct^{\rho-1/2}.
\end{equation}
\end{lemma}

\begin{proof}
We first establish a negative potential behind the front. Choose
$\delta,a_*>0$ so that $f'(s)\le-a_*$ on $[1-\delta,1]$.
Parameter ordering and phase convergence for $\beta_*$ allow a fixed
$K_*>0$ and a common starting time such that
$q_\beta(t,x)\ge1-\delta$ for $0\le x\le m(t)-K_*$ and
$0\le\beta\le\beta_*$. Reflection comparison gives the same lower bound
for $p_\beta$ in this range.
To cover the remaining donor-side region, put $L=-a>0$, where $a$ is from
Assumption~\ref{hyp:data}, and consider the Dirichlet heat solution
\[
 v(t,y)=\int_L^\infty[G_t(y-s)-G_t(y+s)]\,ds,\qquad y>0.
\]
Its initial datum lies below $p_\beta(0,\cdot)$, its boundary value is zero,
and $f(p_\beta)\ge0$, so $p_\beta\ge v$ for every $0\le\beta\le\beta_*$.
For $y\ge L$,
\[
 1-v(t,y)=\frac12\erfc\!\left(\frac{y-L}{2\sqrt t}\right)
          +\frac12\erfc\!\left(\frac{y+L}{2\sqrt t}\right).
\]
Since $m(t)-K_*\sim2t$, this deficit is at most $Ce^{-ct}$ uniformly for
$y\ge m(t)-K_*$ at large times. Thus both $u^\beta$ and $U$ lie in
$[1-\delta,1]$ whenever $x\le m(t)-K_*$. The entire secant segment lies
in that interval, proving
\begin{equation}\label{small:negative-bulk-potential}
 c_\beta(t,x)\le-a_*\qquad(x\le m(t)-K_*),
\end{equation}
for all sufficiently large times, uniformly in $0<\beta\le\beta_*$.

Take an even $\chi\in C_c^\infty(\mathbb R)$ equal to one near zero and set
\begin{equation}\label{small:jump-lift}
 V_\beta=W_\beta-\frac12d_\beta(t)\operatorname{sgn}(x)\chi(x).
\end{equation}
By~\eqref{small:quotient-jump}, $V_\beta$ and its first spatial derivative
are continuous across zero. Its full-space equation holds in distributions:
\begin{equation}\label{small:lifted-equation}
 (V_\beta)_t=(V_\beta)_{xx}+c_\beta V_\beta+\mathcal F_\beta,
 \quad
 \mathcal F_\beta=\frac{\operatorname{sgn}(x)}2
 \bigl(d_\beta\chi''+c_\beta d_\beta\chi-d'_\beta\chi\bigr).
\end{equation}
There are no singular interface terms because the two jumps vanish.
Lemma~\ref{small:uniform-flux} gives $|\mathcal F_\beta|\le Ce^{-\omega t}$,
with support in the fixed compact set $\operatorname{supp}\chi$.

In the coordinate $z=x-m(t)$, set $r_\beta(t,z)=e^{x-2t}V_\beta(t,x)$.
Then
\begin{equation}\label{small:weighted-lift-equation}
 (r_\beta)_t=(r_\beta)_{zz}-\frac3{2t}(r_\beta)_z
 -(1-c_\beta)r_\beta+e^{x-2t}\mathcal F_\beta,
\end{equation}
where the coefficients on the right are evaluated at $x=m(t)+z$.
Choose $\lambda>1$ close enough to one that
$\lambda^2<1+a_*$ and $2(\lambda-1)<\omega$, and define
\begin{equation}\label{small:psi-screen}
 \psi(z)=\begin{cases}
 e^{\lambda(z+K_*)},&z<-K_*,\\
 1+\lambda(z+K_*),&z\ge-K_*.
 \end{cases}
\end{equation}
For $\mathscr Q_\beta=\partial_t-\partial_{zz}
+\frac3{2t}\partial_z+(1-c_\beta)$, we have
$\mathscr Q_\beta\psi\ge0$ at large times. On the left this follows from
\eqref{small:negative-bulk-potential} and $1+a_*-\lambda^2>0$;
on the right it follows from $c_\beta\le1$ and $\psi'\ge0$.
The weight is $C^1$, so its second derivative has no measure at the join.

Choose a common $T$ large enough that these estimates hold and
$\operatorname{supp}\chi\subset(-\infty,m(t)-K_*)$ for $t\ge T$.
On the support of the forcing,
\[
 \sup_x\frac{e^{x-2t}|\mathcal F_\beta(t,x)|}{\psi(x-m(t))}
 \le C e^{-[\omega-2(\lambda-1)]t}t^{-3\lambda/2}.
\]
Accordingly, choose the same bounded increasing function for all parameters,
\[
 A(t)=A_0+C\int_T^t
 e^{-[\omega-2(\lambda-1)]s}s^{-3\lambda/2}\,ds,
\]
with $C$ large enough to dominate this forcing bound.
The Gaussian estimate~\eqref{small:finite-gaussian} and the compact support
of the lifting allow a common $A_0$ such that
$|V_\beta(T,x)|\le A_0e^{2T-x}\psi(x-m(T))$ for all $x$.
To apply comparison, return to the original variables and set
\[
 \mathcal V(t,x)=A(t)e^{2t-x}\psi(x-m(t)).
\]
Then $(\partial_t-\partial_{xx}-c_\beta)\mathcal V\ge|\mathcal F_\beta|$
in distributions. The function $\mathcal V$ is bounded on each strip
$[T,T_1]\times\mathbb R$, $T_1>T$: as $x\to-\infty$ it decays like
$e^{(\lambda-1)x}$, and as $x\to+\infty$ like a linear factor times $e^{-x}$.
The same boundedness holds uniformly in $\beta$ for $V_\beta$ by
\eqref{small:finite-quotient-bound} and~\eqref{small:flux-decay}.
Although $c_\beta$ may jump at zero, it is a bounded measurable coefficient
on the full line. The convex approximation used in the proof of
Corollary~\ref{cor:signed-amplitudes}, now including the source, gives
\[
 (\partial_t-\partial_{xx}-c_\beta)|V_\beta|\le|\mathcal F_\beta|
\]
in distributions. The approximation is valid across zero because
$V_\beta$ and $(V_\beta)_x$ are continuous there, so no interface measure
is generated. To justify full-space comparison, set
$Z=|V_\beta|-\mathcal V$ and apply the positive-part version of the same
inequality. With $M=\|c_\beta\|_\infty\le\|f'\|_\infty$, heat-kernel
testing of the bounded distributional subsolution gives
\[
 \|Z_+(t)\|_\infty\le M\int_T^t\|Z_+(s)\|_\infty\,ds,
 \qquad Z_+(T)=0.
\]
The testing can be performed first with smooth spatial cutoffs and then
on the whole line; boundedness permits removal of the cutoffs against the
Gaussian kernel. Gronwall's inequality yields $Z_+=0$, hence
$|V_\beta|\le\mathcal V$. Since $A$ is uniformly bounded,
\begin{equation}\label{small:weighted-quotient-screen}
 |r_\beta(t,z)|\le C\psi(z).
\end{equation}

We now check the hypotheses of Corollary~\ref{cor:signed-amplitudes}.
Choose a common $T_0\ge T$ large enough that $x=m(t)-t^\rho$ lies to the
right of $\operatorname{supp}\chi$ and $-t^\rho<-K_*$ for $t\ge T_0$.
At this negative cutoff, $V_\beta=W_\beta$ and the normalized trace satisfies
\[
 |Z_{W_\beta}(\log t,0)|
 =|t e^{x-2t}W_\beta(t,x)|
 \le Ct e^{-\lambda t^\rho}
 \le C e^{-(\lambda/2)t^\rho}.
\]
The equality uses $m(t)=2t-\frac32\log t$.
At the fixed time $T_0$,~\eqref{small:finite-gaussian} supplies the common
$X_{1/5}$ bound for $Z_{W_\beta}(\log T_0,\cdot)$ established in the first
subsection. The solutions $W_\beta$ are classical on the right half-line,
continuous up to $T_0$, and uniformly bounded on every finite time strip.
Their coefficients are continuous there, and concavity and $f'(0)=1$ give
\[
 c_\beta\le1,\qquad
 0\le1-c_\beta\le C(u^\beta+U).
\]
With the cutoff $\chi$ fixed once and for all, the quantities
$K_*,\lambda,T,T_0$ and the constants in these bounds can all be chosen
in terms of $f,u_0,\beta_*,\rho$, independently of
$\beta\in(0,\beta_*]$. Thus all hypotheses of the corollary hold with
common constants and with the two background solutions $u^\beta,u^0$.
It yields signed amplitudes $\alpha_{W_\beta}$ and a common moment error.
For each fixed $\beta>0$, linearity of the moment and
Lemma~\ref{lem:leading-edge} identify
\[
 \alpha_{W_\beta}=\frac{\alpha_\beta-\alpha_0}{\beta}.
\]
Substituting this identity in~\eqref{front:signed-moment} proves
\eqref{small:moment-uniform-error}.
\end{proof}

\begin{proposition}\label{small:amplitude-derivative}
Under Assumptions~\ref{hyp:f}--\ref{hyp:concave}, the right derivative
$\dot\alpha_0$ of $\alpha_\beta$ at zero exists and is finite.
For every $0<\theta<1/3$,
\begin{equation}\label{small:response-amplitude}
 \dot\alpha_0=\lim_{t\to\infty}
 \frac{e^{t^\theta}w(t,m(t)+t^\theta)}{t^\theta}.
\end{equation}
Consequently,
\begin{equation}\label{small:delay-expansion}
 \tau_\beta=K\beta+o(\beta)\qquad(\beta\downarrow0),\qquad K=-\frac{\dot\alpha_0}{2\alpha_0}.
\end{equation}
\end{proposition}

\begin{proof}
Fix $0<\rho<1/3$. For each fixed, sufficiently large $t$, finite-time
convergence from Proposition~\ref{small:finite-variation} and the common
Gaussian bound~\eqref{small:finite-gaussian} permit dominated convergence:
\[
 \frac{\mathcal J_\beta(t)-\mathcal J_0(t)}\beta
 \longrightarrow
 \mathcal J_w(t):=e^{-2t}\int_{\ell(t)}^\infty
 (x-\ell(t))e^xw(t,x)\,dx
 \qquad(\beta\downarrow0).
\]
At one fixed such time, Lemma~\ref{small:scaled-moment} therefore bounds
$(\alpha_\beta-\alpha_0)/\beta$ for small $\beta$.
Applying its uniform error estimate at any sufficiently large $t$ gives
\[
 2\sqrt\pi\left(
 \limsup_{\beta\downarrow0}\frac{\alpha_\beta-\alpha_0}{\beta}
 -\liminf_{\beta\downarrow0}\frac{\alpha_\beta-\alpha_0}{\beta}\right)
 \le 2Ct^{\rho-1/2}.
\]
Letting $t\to\infty$ shows that the two finite limits coincide, proving
existence of $\dot\alpha_0$. Passing $\beta\downarrow0$ in the same estimate
now yields the quantitative bound
\begin{equation}\label{small:response-moment-bound}
 |\mathcal J_w(t)-2\sqrt\pi\dot\alpha_0|\le Ct^{\rho-1/2}.
\end{equation}
In particular,
\begin{equation}\label{small:response-moment-limit}
 \dot\alpha_0=\frac1{2\sqrt\pi}\lim_{t\to\infty}\mathcal J_w(t).
\end{equation}

The moment $\mathcal J_w$ uses the fixed cutoff exponent $\rho$, but
\eqref{small:response-moment-limit} identifies its limiting amplitude with
the same derivative $\dot\alpha_0$ for every $0<\rho<1/3$.
The common negative-cutoff trace bound passes to $w$ by pointwise
convergence at each fixed time. At $T_0$, the common Gaussian bound
\eqref{small:finite-gaussian} passes to $w$ as well and gives the initial
$X_{1/5}$ estimate: equivalently, take the pointwise limit in the common
weighted bound for $Z_{W_\beta}$ and then the supremum over $y\ge0$.
Its equation has coefficient $f'(U)$, which is continuous, bounded, and satisfies
\[
 f'(U)\le1,\qquad 0\le1-f'(U)\le\|f''\|_\infty U.
\]
Together with the regularity and finite-strip boundedness from
Proposition~\ref{small:finite-variation}, these verify the hypotheses of
Corollary~\ref{cor:signed-amplitudes} for $w$, using $u^0$ as both background
solutions. The corollary gives an amplitude $\alpha_w$;
\eqref{small:response-moment-limit} identifies it as $\dot\alpha_0$.
For any fixed $0<\theta<1/3$, choose $0<\rho<\theta$ and use
\eqref{front:signed-intermediate} to obtain~\eqref{small:response-amplitude}.
Finally, substituting $\alpha_\beta=\alpha_0+\beta\dot\alpha_0+o(\beta)$
into~\eqref{order:tau-alpha} proves~\eqref{small:delay-expansion}.
\end{proof}

We have established the first-order expansion and identified its coefficient
through the linear variation $w$. The next subsection gives a further linear
characterization and proves a strict lower bound for this coefficient.

\subsection{A strict lower bound and a linear characterization of the coefficient}

We retain Assumptions~\ref{hyp:f}--\ref{hyp:concave} and take all background
quantities from $U=u^0$. The first variation from
Proposition~\ref{small:finite-variation} determines the responses of the
folded mean and difference. For $x>0$, set
\[
 p(t,x)=U(t,-x),\qquad q(t,x)=U(t,x),
\]
and define
\[
 R(t,x)=-\frac{w(t,-x)+w(t,x)}2,
 \qquad H(t,x)=\frac{w(t,-x)-w(t,x)}2.
\]
By~\eqref{small:finite-expansion}, these functions equal
$-\partial_\beta S_\beta|_{0^+}$ and
$\partial_\beta D_\beta|_{0^+}$, respectively, with the folded variables
from~\eqref{eq:fold}. Let
\[
 A=\frac{f'(p)+f'(q)}2,\qquad
 B=\frac{f'(q)-f'(p)}2.
\]
Reflection comparison and concavity give $B\ge0$, and $A,B$ are bounded.
The equations for $w$ on the two half-lines give
\begin{equation}\label{small:cooperative-response}
 \begin{cases}
 R_t=R_{xx}+AR+BH,&x>0,\\
 H_t=H_{xx}+AH+BR,&x>0,\\
 R_x(t,0)=0,\\
 H(t,0)=-\dfrac12U_x(t,0),\\
 R(0,x)=H(0,x)=0.
 \end{cases}
\end{equation}
Indeed, the interface conditions in~\eqref{small:variation-problem} imply
\[
 R_x(t,0)=\frac{w_x(t,0^-)-w_x(t,0^+)}2=0,
 \qquad H(t,0)=-\frac{[w](t)}2=-\frac12U_x(t,0).
\]
The functions $R,H$ are bounded on every finite time strip, classical in
the interior for positive times, and have the displayed boundary traces;
their zero initial data are attained uniformly. These properties also
specify the unique bounded solution of the cooperative linear problem.

\begin{proposition}\label{small:strict-coefficient}
Under Assumptions~\ref{hyp:f}--\ref{hyp:concave}, let $Z$ be the unique
solution, bounded on every finite time strip, of
\begin{equation}\label{small:positive-correction}
 \begin{cases}
 Z_t=Z_{xx}+f'(U)Z,&x>0,\\
 Z(t,0)=R(t,0),&t>0,\\
 Z(0,x)=0,&x>0.
 \end{cases}
\end{equation}
Here the equation holds classically in the interior for positive times,
the boundary condition holds by continuity, and the initial condition is
attained uniformly. Then, for every $0<\theta<1/3$, the limit
\begin{equation}\label{small:zeta-definition}
 \zeta=\lim_{t\to\infty}\frac{e^{t^\theta}Z(t,m(t)+t^\theta)}{t^\theta}
\end{equation}
exists, is independent of $\theta$, is strictly positive, and satisfies
\begin{equation}\label{small:coefficient-formula}
 \dot\alpha_0=-\frac{\alpha_0}{2}-\zeta,
 \qquad K=\frac14+\frac{\zeta}{2\alpha_0}>\frac14.
\end{equation}
\end{proposition}

\begin{proof}
Lemma~\ref{lem:reflection-comparison} gives $U_x(t,0)<0$ for $t>0$.
The positive-semigroup argument in the proof of
Proposition~\ref{prop:pointwise-order}, now with Neumann and Dirichlet
semigroups and nonnegative Dirichlet input, gives $R,H\ge0$ in
\eqref{small:cooperative-response}; the positive
Dirichlet boundary value of $H$ and the strong maximum principle give
$H(t,x)>0$ for $t>0,x>0$.
For each fixed $s>0$, the initial end states and the full-line comparisons
\[
 0\le1-U(s,\cdot)\le G_s*(1-u_0),\qquad
 0\le U(s,\cdot)\le e^sG_s*u_0
\]
give $p(s,x)\to1$ and $q(s,x)\to0$ as $x\to\infty$.
The first comparison uses $f(U)\ge0$, and the second uses $f(U)\le U$.
Since $f'(1)<f'(0)$, we have
$B(s,x)>0$ for sufficiently large $x$. Continuity therefore makes $BH$
strictly positive on a nonempty space-time open set before any prescribed
time $t>0$.
Choose $L\ge\|A\|_\infty$ and write
$R_t=R_{xx}-LR+(A+L)R+BH$. The Neumann heat kernel is strictly positive,
including at its boundary argument. Its Duhamel representation, with
nonnegative sources $(A+L)R$ and $BH$, thus gives $R(t,0)>0$ for every
$t>0$. This argument uses concavity, but does not require strict concavity.

On the right half-line, $-w=R+H$, and
\eqref{small:cooperative-response} shows that $R+H$ solves the scalar
linear equation with potential $A+B=f'(q)=f'(U)$.
The function $k=-U_x/2$ solves the same equation and has boundary value
$-U_x(t,0)/2$. To justify its zero initial datum despite the possible
roughness of $u_0$, recall that $u_0=0$ on $[b,\infty)$ for some $b<0$.
Differentiating the full-line heat representation of $U$ gives, with
$d=-b>0$,
\[
 \|U_x(t,\cdot)\|_{L^\infty(0,\infty)}
 \le Ct^{-1/2}e^{-d^2/(4t)}+C\sqrt t
 \longrightarrow0\qquad(t\downarrow0).
\]
Here the first term bounds the initial heat evolution on the right
half-line, and the second uses boundedness of $f(U)$ and
$\|G'_r\|_{L^1}\le Cr^{-1/2}$.
Consequently, $R+H-k$ is bounded on finite time strips, has zero initial
datum uniformly, and has boundary value $R(t,0)$. Uniqueness for
\eqref{small:positive-correction} identifies
\begin{equation}\label{small:variation-decomposition}
 Z=R+H-k,\qquad w=\frac12U_x-Z.
\end{equation}
This also constructs the solution $Z$ in the stated class. Comparison
and the strong maximum principle, using $R(t,0)>0$, give $Z(t,x)>0$
for all $t>0,x>0$.

The gradient limit in~\eqref{front:intermediate}, with $\beta=0$, gives
\begin{equation}\label{small:background-gradient}
 \lim_{t\to\infty}\frac{e^{t^\theta}U_x(t,m(t)+t^\theta)}{t^\theta}=-\alpha_0
 \qquad(0<\theta<1/3).
\end{equation}
Combining this with~\eqref{small:response-amplitude} and
\eqref{small:variation-decomposition} proves that the limit
in~\eqref{small:zeta-definition} exists and equals
$-\alpha_0/2-\dot\alpha_0$, independently of $\theta$.
To prove that it is strictly positive, we apply
Lemma~\ref{lem:positive-tail} directly to $Z$. The solution is positive,
classical, and bounded on every finite time strip away from $t=0$,
and its coefficient $f'(U)$ is bounded. Moreover,
\eqref{eq:uniform-leading-bound} and $f\in C^2([0,1])$ imply, for all
sufficiently large $t$ and $z\ge0$,
\[
 U(t,m(t)+z)\le Ct e^{-z},\qquad
 f'(U(t,m(t)+z))\ge1-\|f''\|_\infty U(t,m(t)+z)
 \ge1-Ct e^{-z}.
\]
All hypotheses of the lemma hold, so the lower limit of the normalized
$Z$ in~\eqref{small:zeta-definition} is strictly positive. Since its limit
already exists, $\zeta>0$. Substitution into
\eqref{small:delay-expansion} proves~\eqref{small:coefficient-formula}.
\end{proof}

The coefficient is thus determined entirely by the homogeneous background
and the linear problems~\eqref{small:cooperative-response} and
\eqref{small:positive-correction}. Together with
Proposition~\ref{small:amplitude-derivative}, this completes the proof of
Theorem~\ref{thm:small-beta}.
For general nonconcave KPP reactions,
Proposition~\ref{small:finite-variation} still applies, but the bound
$c_\beta\le1$ used in the uniform difference-quotient estimates is no
longer automatic. The present argument therefore does not establish the
far-field first-order expansion in that setting.
We now turn to the opposite regime of large resistance, where the
receiving-side amplitude is determined by the weak input through the interface.

\section{The strong-barrier limit}\label{sec:large-beta}

This section proves Theorem~\ref{thm:large-beta} under
Assumptions~\ref{hyp:f}--\ref{hyp:data}, without Assumption~\ref{hyp:concave}. Set $\varepsilon=\beta^{-1}$,
and retain $p_N,h,I$ from~\eqref{eq:donor}--\eqref{eq:donor-mass}.
We first identify the receiving-side profile generated by the weak input
on a fixed time interval. This profile supplies small initial data for a
lower bound on the amplitude; an adjoint-moment estimate for the sustained
input will then give the matching upper bound.

The folded interface conditions read
\[
 (p_\beta)_x(t,0)=j_\beta(t),\qquad
 (q_\beta)_x(t,0)=-j_\beta(t).
\]
By Lemma~\ref{lem:reflection-comparison}, $j_\beta\ge0$.
Thus the donor-side boundary condition describes a loss, and comparison
with the closed Neumann problem, which has the same initial datum, gives
$p_\beta\le p_N$. Equivalently, the outward derivative of $p_\beta$ at
zero is $-j_\beta\le0$, whereas that of $p_N$ is zero.
Together with reflection comparison, this yields
\begin{equation}\label{large:flux}
 0\le j_\beta(t)=\varepsilon\bigl(p_\beta(t,0)-q_\beta(t,0)\bigr)
 \le\varepsilon h(t)\le\varepsilon.
\end{equation}
Recall that $0<h(t)<1$ for $t>0$ and $0<I<1/2$.

\subsection{Linear input at a fixed time}

For the half-line heat equation with zero initial datum and prescribed
incoming flux $-v_x(t,0)=j(t)$, the boundary response is
$2\int_0^tG_{t-s}(x)j(s)\,ds$. Thus the boundary source kernel is
$2G_t(x)$. Using $f(q_\beta)\le q_\beta$ and~\eqref{large:flux}, linear
comparison gives, for $t>0$ and $x\ge0$,
\begin{equation}\label{large:linear-bound}
 q_\beta(t,x)\le2\varepsilon\int_0^t e^{t-s}G_{t-s}(x)h(s)\,ds
 \le\frac{C\varepsilon e^t}{\sqrt{1+t}}.
\end{equation}
The constant $C$ is independent of $t,x,\beta$. For the last inequality,
use $h\le1$, $G_r(x)\le(4\pi r)^{-1/2}$, and
\[
 \int_0^t\frac{e^{-s}}{\sqrt{t-s}}\,ds\le\frac{C}{\sqrt{1+t}};
\]
the bound follows directly for $t\le1$ and by splitting at $s=t/2$ for
$t\ge1$.

\begin{lemma}
\label{large:fixed-time-input}
Under Assumptions~\ref{hyp:f}--\ref{hyp:data}, define
\[
 v(t,x)=2\int_0^t e^{t-s}G_{t-s}(x)h(s)\,ds,\qquad v(0,x)=0.
\]
For every fixed $T>0$, there is $C_T$, depending only on $f,u_0,T$,
such that, for $0<\varepsilon\le1$ and $\beta=\varepsilon^{-1}$,
\begin{equation}\label{large:finite-time}
 \sup_{0\le t\le T}
 \left(\|p_N(t)-p_\beta(t)\|_\infty+\|q_\beta(t)\|_\infty\right)
 \le C_T\varepsilon,
\end{equation}
\begin{equation}\label{large:linear-input-error}
 \sup_{0\le t\le T}
 \left(\left|\frac{j_\beta(t)}\varepsilon-h(t)\right|
 +\left\|\frac{q_\beta(t)}\varepsilon-v(t)\right\|_\infty\right)
 \le C_T\varepsilon.
\end{equation}
Here $\|\cdot\|_\infty$ denotes the $L^\infty(0,\infty)$ norm.
At positive times we use the continuous representatives, so the bounds
also hold for the boundary traces. The flux is extended to $t=0$ by
$j_\beta(0)=0$.
In particular,
\begin{equation}\label{large:seed}
 \varepsilon^{-1}q_\beta(T,\cdot)\longrightarrow
 v_T:=v(T,\cdot)
 \qquad\text{uniformly on }[0,\infty).
\end{equation}
The profile $v_T$ is continuous and strictly positive on $[0,\infty)$.
After even extension to $\mathbb R$, it has the exponential moment
\begin{equation}\label{large:seed-mass}
 e^{-2T}\int_\mathbb R e^xv_T(x)\,dx
 =2\int_0^T e^{-2s}h(s)\,ds=:2I_T.
\end{equation}
\end{lemma}

\begin{proof}
Recall the Neumann heat semigroup $\mathsf N(t)$ on the half-line and put
\[
 b_\beta(t,x)=2\int_0^tG_{t-s}(x)j_\beta(s)\,ds.
\]
The mild formulas with boundary loss and input give
\[
 \begin{split}
 p_N(t)-p_\beta(t)
 &=b_\beta(t)+\int_0^t\mathsf N(t-s)
       [f(p_N(s))-f(p_\beta(s))]\,ds,\\
 q_\beta(t)&=b_\beta(t)+\int_0^t\mathsf N(t-s)f(q_\beta(s))\,ds.
 \end{split}
\]
Since $\|b_\beta(t)\|_\infty\le2\varepsilon\sqrt{t/\pi}$,
contractivity of $\mathsf N$ and the Lipschitz continuity of $f$, with
$f(0)=0$, imply
\[
 \mathcal E_\beta(t):=\|p_N(t)-p_\beta(t)\|_\infty+\|q_\beta(t)\|_\infty
 \le C\varepsilon\sqrt t+\|f'\|_\infty\int_0^t \mathcal E_\beta(s)\,ds.
\]
Gronwall's inequality proves~\eqref{large:finite-time}.
These formulas hold up to the boundary for positive times. The initial
profiles vanish near the interface, so the heat representations also give
$p_\beta(t,0),q_\beta(t,0),h(t)\to0$ as $t\downarrow0$.
Thus $j_\beta$ has the stated continuous extension at zero. The identity
$j_\beta/\varepsilon=p_\beta(t,0)-q_\beta(t,0)$ now gives
\[
 \sup_{0\le t\le T}
 \left|\frac{j_\beta(t)}\varepsilon-h(t)\right|\le C_T\varepsilon.
\]

Writing $f(q)=q-g(q)$, the receiving-side formula becomes
\[
 q_\beta(t,x)=2\int_0^t e^{t-s}G_{t-s}(x)j_\beta(s)\,ds
 -\int_0^t e^{t-s}
      \bigl[\mathsf N(t-s)g(q_\beta(s))\bigr](x)\,ds.
\]
By~\eqref{eq:g} and~\eqref{large:finite-time},
$\sup_{s\le T}\|g(q_\beta(s))\|_\infty\le C_T\varepsilon^2$.
Divide the formula by $\varepsilon$ and subtract $v(t,x)$.
The flux error contributes at most
$C_T\varepsilon\int_0^t(t-s)^{-1/2}\,ds$, and the nonlinear term
contributes at most $C_T\varepsilon\int_0^t e^{t-s}\,ds$.
This proves~\eqref{large:linear-input-error} and hence~\eqref{large:seed}.

Continuity and positivity of $v_T$ follow from its integral representation,
$h(s)>0$ for $s>0$, and the integrable bound $G_r(x)\le Cr^{-1/2}$.
After even extension, Tonelli's theorem and
$\int_\mathbb R e^xG_r(x)\,dx=e^r$ give
\[
 e^{-2T}\int_\mathbb R e^xv_T(x)\,dx
 =2e^{-2T}\int_0^T e^{2(T-s)}h(s)\,ds=2I_T,
\]
which proves~\eqref{large:seed-mass}.
\end{proof}

The error bounds above hold on fixed time intervals, with constants that
may grow with $T$. We next use compactly supported profiles below $v_T$
to construct a lower comparison.

\subsection{A lower bound from small initial data}

We convert the fixed-time input profile into a lower bound on the
leading-edge amplitude using the following small-data law for the logistic
equation.

\begin{lemma}\label{large:small-data}
Let $\chi\in C_c(\mathbb R)$ be fixed, nonzero, and nonnegative. For
$\delta>0$ small enough that $\delta\|\chi\|_\infty\le1$, let $v^\delta$
solve the full-line problem
\[
 v_t=v_{xx}+v-v^2,\qquad v(0,x)=\delta\chi(x).
\]
Normalize the minimal-speed logistic wave by
$U_{\rm log}(z)\sim ze^{-z}$ as $z\to+\infty$. This fixes its tail coefficient at one, rather than imposing $U_{\rm log}(0)=1/2$. Denote the displacement parameter of the right-moving front by $x_\delta$, using the sign convention of~\cite{MRR2022}, so that
\[
 v^\delta(t,m(t)+z)\longrightarrow U_{\rm log}(z+x_\delta)
 \qquad(t\to\infty)
\]
locally uniformly in $z$. In the phase convention $U(z-\xi)$ used above, the corresponding phase is $\xi_{{\rm log},\delta}=-x_\delta$. Define its leading-edge amplitude to be the tail coefficient of this limiting profile,
$\alpha_{\rm log}(\delta\chi):=e^{-x_\delta}$. Then
\begin{equation}\label{large:small-data-amplitude}
 \alpha_{\rm log}(\delta\chi)
 \sim\frac{\delta\log\delta^{-1}}{2\sqrt\pi}
 \int_\R e^x\chi(x)\,dx\qquad(\delta\downarrow0).
\end{equation}
\end{lemma}

\begin{proof}
With this wave normalization and sign convention for the shift,
Mytnik--Roquejoffre--Ryzhik~\cite[Theorem 1.1]{MRR2022} give
\[
 x_\delta=\log\delta^{-1}-\log\log\delta^{-1}
 -\log\left(\frac1{2\sqrt\pi}\int_\R e^x\chi(x)\,dx\right)+o(1).
\]
Our fixed, nonnegative, continuous, compactly supported profile satisfies
their hypotheses. Exponentiating the negative of this expansion proves
\eqref{large:small-data-amplitude}.
\end{proof}

\begin{proof}[Lower bound for the amplitude in Theorem~\ref{thm:large-beta}]
By~\eqref{eq:g}, choose $C_f\ge1$, depending only on $f$, such that
$f(s)\ge s-C_fs^2$ on $[0,1]$.
Fix $T>0$ and choose a nonzero, nonnegative, even function
$\chi\in C_c(\mathbb R)$ such that, for some $0<\eta<1$,
\[
 \chi(x)\le(1-\eta)v_T(x)\qquad(x\in\operatorname{supp}\chi),
\]
where $v_T$ is extended evenly. Since $v_T$ is continuous and strictly
positive, the gap $v_T-\chi$ on this compact support is at least
$\eta\min_{\operatorname{supp}\chi}v_T>0$.
The uniform convergence in~\eqref{large:seed} therefore gives
$\varepsilon\chi\le q_\beta(T)$ on $x\ge0$ for all sufficiently small
$\varepsilon$, with the threshold allowed to depend on $T,\chi,\eta$.
Outside the support this comparison follows from $q_\beta\ge0$.
Let $w^\varepsilon$ solve the full-line problem
\[
 w_t=w_{xx}+w-C_fw^2,\qquad w(0,x)=\varepsilon\chi(x).
\]
Further restrict $\varepsilon$ so that
$\varepsilon\|\chi\|_\infty\le1/C_f$. Comparison with the constant
solutions gives
\[
 0\le w^\varepsilon(t,x)\le1/C_f\le1\qquad(t\ge0,x\in\mathbb R).
\]
Hence $w^\varepsilon_t-w^\varepsilon_{xx}
=w^\varepsilon-C_f(w^\varepsilon)^2\le f(w^\varepsilon)$ throughout
its evolution. Evenness and uniqueness imply
$w^\varepsilon_x(s,0)=0$ for $s>0$, whereas
\[
 -w^\varepsilon_x(s,0)=0\le j_\beta(T+s)
 =-q_{\beta,x}(T+s,0).
\]
The half-line comparison principle, with the initial ordering at time $T$,
therefore yields
\begin{equation}\label{large:lower-comparison}
 q_\beta(t,x)\ge w^\varepsilon(t-T,x)\qquad(t\ge T,x\ge0).
\end{equation}

The function $C_fw^\varepsilon$ solves the standard logistic equation
with initial datum $C_f\varepsilon\chi$. Write
$x_\varepsilon^{\rm log}$ for its displacement parameter in the convention $U_{\rm log}(z+x_\varepsilon^{\rm log})$ of Lemma~\ref{large:small-data}, and define
\[
 \alpha[w^\varepsilon]
 :=C_f^{-1}e^{-x_\varepsilon^{\rm log}}
 =C_f^{-1}\alpha_{\rm log}(C_f\varepsilon\chi).
\]
Lemma~\ref{large:small-data}, applied with $\delta=C_f\varepsilon$, gives
\[
 \alpha[w^\varepsilon]
 \sim\frac{\varepsilon\log\varepsilon^{-1}}{2\sqrt\pi}
       \int_\R e^x\chi(x)\,dx.
\]
Indeed, the multiplicative $C_f$ cancels, and
$\log(C_f\varepsilon)^{-1}
=\log\varepsilon^{-1}-\log C_f\sim\log\varepsilon^{-1}$.

For fixed $\varepsilon,T,\chi$, evaluate~\eqref{large:lower-comparison}
at $x=m(t)+z$. Since
\[
 x-m(t-T)=z+m(t)-m(t-T)=z+2T+o(1),
\]
phase convergence for the actual solution and for the logistic comparison
solution gives, for every fixed $z\in\mathbb R$,
\[
 U_*(z-\xi_\beta)
 \ge C_f^{-1}U_{\rm log}(z+2T+x_\varepsilon^{\rm log}).
\]
Now let $z\to+\infty$. The coefficient of $ze^{-z}$ on the left is
$\alpha_\beta$, and that on the right is
$e^{-2T}\alpha[w^\varepsilon]$. Thus
\[
 \alpha_\beta\ge e^{-2T}\alpha[w^\varepsilon].
\]
This comparison uses the tails of the two limiting waves and does not
require their profiles to coincide. Letting $\varepsilon\downarrow0$
with $T,\chi$ fixed gives
\begin{equation}\label{large:truncated-lower-bound}
 \liminf_{\beta\to\infty}\frac{\beta\alpha_\beta}{\log\beta}
 \ge\frac{e^{-2T}}{2\sqrt\pi}\int_\mathbb R e^x\chi(x)\,dx.
\end{equation}

To recover the full profile, choose continuous even cutoffs
$0\le\eta_n\le1$ with compact support, $\eta_n=1$ on $[-n,n]$,
and $\eta_n\uparrow1$ pointwise. For $n\ge2$, set
\[
 \chi_n=(1-1/n)\eta_n v_T.
\]
Each $\chi_n$ is admissible with gap parameter $\eta=1/n$, and
$\chi_n\uparrow v_T$. Apply~\eqref{large:truncated-lower-bound} first
for each fixed $T,n$, and then let $n\to\infty$.
Monotone convergence and~\eqref{large:seed-mass} yield
\[
 \liminf_{\beta\to\infty}\frac{\beta\alpha_\beta}{\log\beta}
 \ge\frac{I_T}{\sqrt\pi}.
\]
No uniform small-data asymptotics in $n$ or $T$ are required: the
$\varepsilon\downarrow0$ limit was taken with both fixed.
Finally, $I_T\uparrow I$ as $T\to\infty$, proving
\begin{equation}\label{large:lower-limit}
 \liminf_{\beta\to\infty}\frac{\beta\alpha_\beta}{\log\beta}
 \ge\frac I{\sqrt\pi}.
\end{equation}
\end{proof}

The lower bound uses only the population supplied by a fixed time $T$,
after which the comparison solution evolves without further input.
To obtain the matching upper bound, we next control the contribution of
the entire boundary input through a weighted moment estimate.

\subsection{An adjoint moment upper bound for the sustained input}

Extend the receiving-side solution evenly:
$Q_\beta(t,X)=q_\beta(t,|X|)$. Its initial datum is zero, and
\begin{equation}\label{large:point-source}
 (Q_\beta)_t=(Q_\beta)_{XX}+f(Q_\beta)+2j_\beta(t)\delta_0,
 \qquad 0\le Q_\beta\le1,
\end{equation}
in distributions on the full line. Indeed,
$(Q_\beta)_X(t,0^+)-(Q_\beta)_X(t,0^-)=-2j_\beta(t)$, so the positive
source cancels the singular part of the distributional second derivative.

Set $L=\log\beta\ge2$ and $k=3/2$. We shift the observation frame backwards
by $L$ to obtain a rescaled amplitude $\beta\alpha_\beta$; the source
position in this frame will contribute a first moment of order $L$.
Using $1+t$ makes the coordinates regular at the initial time. Define
\[
 \mu(t)=2t-k\log(1+t),\qquad \sigma=\log(1+t),\qquad
 X=\mu(t)-L+\eta\sqrt{1+t},
\]
and
\begin{equation}\label{large:rescaling}
 r(\sigma,\eta)=(1+t)^{-1/2}e^{\eta\sqrt{1+t}}Q_\beta(t,X)
 =\beta(1+t)e^{X-2t}Q_\beta(t,X).
\end{equation}
We restrict this rescaled solution to $\eta>0$. Its equation is
\begin{equation}\label{large:r-equation}
 r_\sigma=r_{\eta\eta}
 +\left(\frac\eta2-ke^{-\sigma/2}\right)r_\eta+r-N+\nu_\beta,
\end{equation}
where $N=\sqrt{1+t}\,e^{\eta\sqrt{1+t}}g(Q_\beta(t,X))\ge0$, and
\begin{equation}\label{large:rescaled-source}
 \nu_\beta=2\frac{j_\beta(t)}\varepsilon(1+t)^{3/2}e^{-2t}
 \delta_{\eta_s(t)},\qquad
 \eta_s(t)=\frac{L-\mu(t)}{\sqrt{1+t}}.
\end{equation}
The source scaling uses
$\delta_0(X)=(1+t)^{-1/2}\delta_{\eta_s(t)}$.
Equation~\eqref{large:r-equation} holds in distributions on $\eta>0$;
the source is present there only when $\eta_s(t)>0$.
In the quantitative mass, boundary, and moment bounds below, $C$ depends
only on $f,u_0$ and is independent of $\beta\ge e^2$ and time.
For the technical justification of the moment identity on finite time
intervals, constants may depend on the fixed parameter and time horizon;
we indicate these dependencies explicitly.

The weight $\eta$ lies in the kernel of the formal adjoint of the
self-similar operator without the decaying drift term:
\[
 \left(\partial_{\eta\eta}-\frac\eta2\partial_\eta+\frac12\right)\eta=0.
\]
To cancel the leading error due to the drift $-ke^{-\sigma/2}$, introduce
\begin{equation}\label{large:adjoint-weight}
 \psi(\eta)=\int_0^\eta e^{z^2/4}
 \left(\int_z^\infty e^{-y^2/4}\,dy\right)dz,
 \qquad W(\sigma,\eta)=\eta+ke^{-\sigma/2}\psi(\eta),\quad \eta\ge0.
\end{equation}
Direct differentiation and the Gaussian tail bound give
\begin{equation}\label{large:weight-bounds}
 \psi''-\frac\eta2\psi'=-1,\qquad
 0\le\psi'\le\sqrt\pi,\qquad
 0\le\psi(\eta)\le C\log(2+\eta).
\end{equation}
Consequently,
\begin{equation}\label{large:adjoint-residual}
 W_\sigma+W_{\eta\eta}
 -\left(\frac\eta2-ke^{-\sigma/2}\right)W_\eta+\frac12W
 =k^2e^{-\sigma}\psi'.
\end{equation}
This is the approximate adjoint weight of
\cite[Lemma 4.1]{MRR2022}; the nonzero-boundary contribution is also treated
in Section~6 of that work. Here we additionally account for the moving
point source generated by the actual interface flux.

\begin{lemma}
\label{large:moment-bound}
Under Assumptions~\ref{hyp:f}--\ref{hyp:data}, let $\beta\ge e^2$ and
use the rescaling and weight above. The moment
\[
 \mathcal M(\sigma)=\int_0^\infty W(\sigma,\eta)r(\sigma,\eta)\,d\eta
\]
is finite and locally absolutely continuous for $\sigma\ge0$, with
$\mathcal M(0)=0$. There is a constant $C$, depending only on $f,u_0$,
such that, for all $\sigma\ge0$,
\begin{equation}\label{large:uniform-first-moment}
 \int_0^\infty\eta r(\sigma,\eta)\,d\eta
 \le\mathcal M(\sigma)\le2IL+C\log(2+L)+C.
\end{equation}
\end{lemma}

\begin{proof}
We first bound the mass and boundary trace. Full-space linear comparison
in~\eqref{large:point-source} gives
\[
 Q_\beta(t,X)\le2\int_0^t e^{t-s}G_{t-s}(X)j_\beta(s)\,ds.
\]
Changing variables in the mass integral, enlarging its spatial domain to
$\mathbb R$, and using $\int_\mathbb R e^XG_r(X)\,dX=e^r$ yield
\begin{equation}\label{large:mass-bound}
 \int_0^\infty r(\sigma,\eta)\,d\eta
 \le 2\sqrt{1+t}\int_0^t\frac{j_\beta(s)}\varepsilon e^{-2s}\,ds
 \le2I e^{\sigma/2}.
\end{equation}
For $b(\sigma)=r(\sigma,0)$,~\eqref{large:linear-bound} implies
$b\le C\varepsilon e^t/(1+t)\le Ce^{-3L/4}$ when $t\le L/4$.
For $t\ge L/4$, the bound $Q_\beta\le1$ gives
$b\le(1+t)^{-1/2}$. Since $d\sigma=dt/(1+t)$,
\begin{equation}\label{large:boundary-bound}
 \int_0^\infty b(\sigma)\,d\sigma
 \le Ce^{-3L/4}\log(1+L)+CL^{-1/2}\le CL^{-1/2}.
\end{equation}

Define the source weight without evaluating $W$ at negative arguments:
\[
 \mathcal W_s(\sigma)=
 \begin{cases}
 W(\sigma,\eta_s(t)),&\eta_s(t)>0,\\
 0,&\eta_s(t)\le0,
 \end{cases}
 \qquad t=e^\sigma-1.
\]
Using~\eqref{large:adjoint-residual} and $W(\sigma,0)=0$, integration
by parts gives, for almost every $\sigma>0$,
\begin{equation}\label{large:moment-identity}
 \begin{split}
 \mathcal M'={}&W_\eta(\sigma,0)b(\sigma)
 +k^2e^{-\sigma}\int_0^\infty\psi' r\,d\eta
 -\int_0^\infty WN\,d\eta\\
 &+2\frac{j_\beta(t)}\varepsilon(1+t)^{3/2}e^{-2t}
 \mathcal W_s(\sigma).
 \end{split}
\end{equation}
We justify this identity in its time-integrated form. First fix $\beta$
and a finite physical time horizon $T$. The linear Gaussian bound and
$j_\beta\le\varepsilon$ give
\[
 \int_{\mathbb R}(1+|X|)e^X Q_\beta(t,X)\,dX
 \le C_T\varepsilon t,\qquad 0\le t\le T,
\]
because $\int_{\mathbb R}(1+|X|)e^XG_s(X)\,dX$ is bounded for
$0<s\le T$. On $\eta\ge0$, we have $W\le C(1+\eta)$, and the
coordinate change therefore implies
\[
 0\le\mathcal M(\log(1+t))\le C_{\beta,T}t\longrightarrow0
 \qquad(t\downarrow0).
\]
This establishes the initial weighted trace. These finite-time constants
are used only to justify the identity, not to obtain its uniform upper bound.

Choose $\vartheta\in C_c^\infty([0,\infty))$ with
$0\le\vartheta\le1$, $\vartheta=1$ on $[0,1]$, and $\vartheta=0$ on
$[2,\infty)$, and put $\vartheta_R(\eta)=\vartheta(\eta/R)$.
Use $W\vartheta_R$ as a spatial test function, together with a smooth
time test function, in~\eqref{large:r-equation}. Its value at zero
vanishes and its derivative there equals $W_\eta(\sigma,0)$, giving
the displayed boundary term. The distributional source gives its weight
$\mathcal W_s$ once $R$ exceeds the source position on the time interval.
All spatial derivatives can be transferred to the test function. Thus
the cutoff errors involve only $r$ multiplied by derivatives of
$W\vartheta_R$, supported in $R\le\eta\le2R$. The Gaussian bound
controls these terms on the fixed finite time interval and makes them
vanish as $R\to\infty$; no far-field derivative estimate for $r$ is needed.
Moreover, $0\le g(Q_\beta)\le C Q_\beta$ gives $N\le C(1+t)r$, so the
weighted absorption is locally integrable by the preceding moment bound.
The boundary and source terms are also locally integrable. Passing to the
limit yields the time-integrated identity, local absolute continuity of
$\mathcal M$, and~\eqref{large:moment-identity} almost everywhere.
The initial weighted trace permits integration from $\sigma=0$.
The function $\mu$ is strictly increasing from zero, so the source crosses
$\eta=0$ once. Since $W(\sigma,0)=0$ and $r$ is continuous across the
source, the crossing produces no additional atomic term in the moment.

We now integrate~\eqref{large:moment-identity} and discard its nonpositive
absorption term. The boundary contribution satisfies
\[
 \int_0^\infty W_\eta(\sigma,0)b(\sigma)\,d\sigma
 \le CL^{-1/2},
\]
because $W_\eta(\sigma,0)=1+k\sqrt\pi e^{-\sigma/2}$ is bounded.
The adjoint error is integrable uniformly in $\beta$: by
\eqref{large:mass-bound} and~\eqref{large:weight-bounds},
\[
 k^2e^{-\sigma}\int_0^\infty\psi' r\,d\eta
 \le C e^{-\sigma/2},\qquad
 \int_0^\infty C e^{-\sigma/2}\,d\sigma\le C.
\]
For the source term, when $\mu(t)<L$ we have
\[
 \sqrt{1+t}\,W(\sigma,\eta_s(t))
 =L-\mu(t)+k\psi(\eta_s(t)).
\]
Its full time integral is therefore
\begin{equation}\label{large:source-moment}
 2\int_{\{t\ge0:\,\mu(t)<L\}}\frac{j_\beta(t)}\varepsilon e^{-2t}
 \left[L-\mu(t)+k\psi(\eta_s(t))\right]dt.
\end{equation}
Here $\mu(t)\ge0$ and $0\le\eta_s(t)\le L$ on the integration region.
Using $j_\beta/\varepsilon\le h$ and~\eqref{large:weight-bounds}, the
contribution of $L-\mu(t)$ is at most $2IL$, and that of the correction
is at most $CI\log(2+L)$. Thus~\eqref{large:source-moment} is bounded by
$2IL+C\log(2+L)$.
All three retained contributions are nonnegative, so their full time
integrals bound their integrals up to any $\sigma$. Together with
$\mathcal M(0)=0$ and $W\ge\eta$, these estimates prove
\eqref{large:uniform-first-moment}.
\end{proof}

\begin{proof}[Upper bound and conclusion of Theorem~\ref{thm:large-beta}]
Fix $\beta\ge e^2$, so that $L$ is fixed, and take $t\to\infty$ first.
For every fixed $\eta>0$, set
\[
 \widehat\eta_t
 =\frac{\mu(t)-L+\eta\sqrt{1+t}-m(t)}{\sqrt t}
 \longrightarrow\eta.
\]
Since $\mu(t)-m(t)\to0$, the point
$X=\mu(t)-L+\eta\sqrt{1+t}$ is eventually positive. The normalized
solution $w_\beta$ from Section~\ref{front:section} satisfies the exact
relation
\[
 r(\sigma,\eta)=\beta\frac{1+t}{t}
 w_\beta(\log t,\widehat\eta_t).
\]
The profile estimate~\eqref{front:nonlinear-profile}, applied near the
positive limit $\eta$, therefore gives
\[
 r(\sigma,\eta)\longrightarrow
 \beta\alpha_\beta\eta e^{-\eta^2/4}.
\]
Apply Fatou's lemma to the nonnegative functions $\eta r(\sigma,\eta)$.
Since $\int_0^\infty\eta^2e^{-\eta^2/4}\,d\eta=2\sqrt\pi$,
Lemma~\ref{large:moment-bound} yields
\[
 2\sqrt\pi\,\beta\alpha_\beta
 \le\liminf_{\sigma\to\infty}\int_0^\infty\eta r(\sigma,\eta)\,d\eta
 \le2IL+C\log(2+L)+C.
\]
Only now divide by $L=\log\beta$ and let $\beta\to\infty$. We obtain
\begin{equation}\label{large:upper-limit}
 \limsup_{\beta\to\infty}\frac{\beta\alpha_\beta}{\log\beta}
 \le\frac I{\sqrt\pi}.
\end{equation}
The lower bound~\eqref{large:lower-limit} and this upper bound prove
\eqref{eq:large-amplitude}. Finally, the amplitude--delay relation
\eqref{order:tau-alpha} gives
\[
 \tau_\beta=\frac12\log\frac{\alpha_0}{\alpha_\beta}
 =\frac12\log\beta-\frac12\log\log\beta
 +\frac12\log\frac{\sqrt\pi\alpha_0}{I}+o(1),
\]
which is~\eqref{eq:large-delay} and completes the proof.
\end{proof}


\end{document}